\documentclass[a4paper,twoside]{article} 
\usepackage[T1]{fontenc}
\usepackage[english]{babel}

\rmfamily
\usepackage{layaureo} 	

\usepackage{enumitem}
\setlist{labelindent=\parindent} 
\setlist[enumerate]{label=\textit{(\roman*)},ref=(\roman*),align=left, leftmargin=*}

\usepackage{amsthm} 
	\newtheoremstyle{thm} 
  	{11pt}
  	{11pt}
 	{\itshape}
  	{ }
  	{\scshape}
 	{.}
  	{ }
  	{ }

	\newtheoremstyle{definition} 
  	{11pt}
  	{11pt}
 	{\itshape}
  	{ }
  	{\scshape}
 	{.}
  	{ }
  	{ }

	\newtheoremstyle{remark} 
 	{6pt}
  	{6pt}
  	{\itshape}
  	{}
  	{\scshape}
  	{.}
  	{ }
  	{ }
	
	\newtheoremstyle{example} 
  	{6pt}
  	{6pt}
  	{ }
  	{}
  	{\scshape}
  	{.}
  	{ }
 	{ }
	
	\swapnumbers
	\theoremstyle{thm}
	\newtheorem{teor}{- Theorem}[section]
	\newtheorem{thm}{- Theorem}[section]
	\newtheorem{prop}[teor]{- Proposition}
	\newtheorem{cor}[teor]{- Corollary}
	\newtheorem{lemma}[teor]{- Lemma}

	\theoremstyle{definition}
	\newtheorem{dfn}[teor]{- Definition}
	
	\theoremstyle{remark}
	
	\newtheorem{oss}[teor]{- Remark}
	
	\theoremstyle{example}
	\newtheorem{exs}[teor]{- Examples}
\usepackage{amsmath,amssymb,mathrsfs,mathtools,braket,textcomp,bm} 
\numberwithin{equation}{section}
\newcommand{\call}[1]{\mathscr{#1}}

\newcommand{\co}{\mathsf{c}}
\newcommand{\eps}{\varepsilon}
\renewcommand{\epsilon}{\eps}
\renewcommand{\phi}{\varphi}
\renewcommand{\theta}{\vartheta}

\newcommand{\N}{\mathbb{N}}

\newcommand{\R}{\mathbb{R}}

\newcommand{\Rd}{\R^d}

\newcommand{\diverge}[1]{#1\to\infty}
\newcommand{\limsucc}[1]{\lim_{\diverge{#1}}}

\newcommand{\de}{\mathrm{d}}
\newcommand{\der}[2]{\frac{\de#1}{\de#2}}

\newcommand{\grad}{\nabla}

\newcommand{\D}{\mathrm{D}}

\newcommand{\abs}[1]{\left\vert#1\right\vert}
\newcommand{\vass}[1]{\abs{~#1~}}
\newcommand{\norm}[1]{\left\Vert#1\right\Vert}
\renewcommand{\set}[1]{\left\{\hspace{-0.1cm}~#1~\hspace{-0.1cm}\right\}}
\newcommand{\symdif}{\!\bigtriangleup\!}

\newcommand{\eqdef}{\coloneqq}

\DeclareMathOperator{\Per}{Per}
\newcommand{\PerK}{\Per_K}

\usepackage{authblk}
\usepackage{comment}
\usepackage{hyperref}

\usepackage[autostyle,babel]{csquotes}
\usepackage[style=numeric,hyperref,backend=bibtex]{biblatex}
\bibliography{nonloc-per}{}

\title{On the asymptotic behaviour of~nonlocal~perimeters}

\author[ ]{\textsc{Judith Berendsen}
	
	{\small Institut f\"ur Analysis und Numerik, Westf\"alische Wilhelms-Universit\"at M\"unster
	
	Einsteinstr. 62, D 48149 M\"unster, Germany
	
	\href{mailto:judith.berendsen@wwu.de}{judith.berendsen@wwu.de}
	}
	
	\vspace{6pt}
	\textsc{Valerio Pagliari}
	
	{\small Dipartimento di Matematica, Universit\`a di Pisa
	
	Largo Bruno Pontecorvo 5, 56127 Pisa, Italy
	
	\href{mailto:pagliari@mail.dm.unipi.it}{pagliari@mail.dm.unipi.it}
	}}

\date{\today}

\begin{document}
\maketitle

\begin{center}
	
\end{center}

\begin{abstract}
	We study a class of integral functionals known as nonlocal perimeters,
	which, intuitively, express a weighted interaction between a set and its complement.
	The weight is provided by a positive kernel $K$,
	which might be singular.
	
	In the first part of the paper,
	we show that these functionals are indeed perimeters in an generalised sense
	and we establish existence of minimisers for the corresponding Plateau's problem;
	also, when $K$ is radial and strictly decreasing,
	we prove that halfspaces are minimisers
	if we prescribe ``flat'' boundary conditions.
	
	A $\Gamma$-convergence result is discussed in the second part of the work.
	We study the limiting behaviour of the nonlocal perimeters
	associated with certain rescalings of a given kernel
	that has faster-than-$L^1$ decay at infinity
	and we show that the $\Gamma$-limit is the classical perimeter,
	up to a multiplicative constant that we compute explicitly.
\end{abstract}

\tableofcontents

\section{Introduction} In a qualitative way, we might think of
the perimeter of a set in the Euclidean space $\Rd$
as a measure of the locus that circumscribes the set itself.
This intuition is captured by the analytical theory
of finite perimeters sets in use nowadays,
which is grounded on Caccioppoli's seminal works and 
on the ideas developed by De Giorgi in the 1950's.
In a sloppy manner,
we may summarise as follows the gist of this theory:
identify a set with its characteristic function,
consider the distributional gradient of the latter
and define its total variation as perimeter of the given set.
A fundamental result by De Giorgi and Federer shows that if we retain this definition
the perimeter coincides with the $(d-1)$-dimensional Hausdorff measure
of a certain subset of the topological boundary,
so that consistency with the na\"\i ve idea is guaranteed.
Besides, the class of sets such that their perimeter is finite has good compactness properties,
thus it is possible to tackle various problems
that are formulated in geometric terms
via the direct method of calculus of variations;
among all the possible examples of this that could be listed,
we cite only Plateau's problem, because
we shall deal with it later on (see Theorem \ref{stm:plateau}).

Going beyond this by now well-established theory,
recently several authors have grown interested in some set functionals
that are globally referred to as \emph{nonlocal perimeters}:
for instance, a prominent case is offered by fractional perimeters
that were introduced by Caffarelli, Roquejoffre and Savin in \cite{crs:nonlocal}
and that were later extended and largely investigated 
(see for instance
\cite{adm:gammaconvergence,csv:quantitativeflatness,cv:uniformestimates,lud:anisotropicfractional,val:afractional}).
The study of nonlocal perimeters is motivated by both theory and application,
as described in the brief account given by Cinti, Serra and Valdinoci in \cite{csv:quantitativeflatness}.
Moreover, although the definition of these functionals might seem distant from De Giorgi's one
(confront \eqref{eq:PerK} and \eqref{eq:DGper}),
nonlocal perimeters resemble the classical one from various perspectives.
Actually, one can prove that they are indeed perimeters
in the sense proposed in \cite{cmp:nonlocalcurvature},
where Chambolle, Morini and Ponsiglione collect some properties
that a set functional should have in order to deserve such label;
up to minor changes, the axiomatic definition they propose is this one:

	\begin{dfn}\label{dfn:per}
	Let $\call{M}$ be the collection of all Lebesgue measurable sets in $\Rd$
	and let $\Omega\in\call{M}$ be a fixed set with strictly positive Lebesgue measure.
	Choose arbitrarily $E,F\in\call{M}$.
	A functional $p_\Omega\colon \call{M}\to[0,+\infty]$ is a \emph{perimeter in $\Omega$} if
		\begin{enumerate}
			\item\label{stm:per1} $p_\Omega(\emptyset)=0$;
			\item\label{stm:per2}  $p_\Omega(E)=p_\Omega(F)$ whenever $\abs{(E\symdif F)\cap\Omega}=0$;
			\item\label{stm:per3}  it is invariant under translations, that is $p_{\Omega+h}(E+h)=p_\Omega(E)$ for any $h\in\Rd$;
			\item\label{stm:per4}  it is finite on any set that is the closure of an open set with compact $C^2$ boundary;
			\item\label{stm:per5}  it is lower semicontinuous w.r.t. $L^1_\mathrm{loc}(\Rd)$-convergence;
			\item\label{stm:per6}  it is submodular, that is
				\begin{equation}
				p_\Omega(E\cap F) + p_\Omega(E\cup F) \leq p_\Omega(E) + p_\Omega(F).
				\end{equation}
		\end{enumerate}
	\end{dfn}
	
The authors also provide examples of functionals
that fit in this framework;
naturally, De Giorgi's perimeter is one of them,
but we also find the functional \eqref{eq:defPerK-Rd},
which stands as an instance of the analysis we carry out here.
Indeed, this paper is devoted to the study of functionals of the form
	\begin{equation}\label{eq:PerK}
	\begin{split}
	\PerK(E,\Omega) \eqdef& \int_{E\cap\Omega}\int_{E^\co\cap\Omega}K(y-x)\de y\de x \\
							& + \int_{E\cap\Omega}\int_{E^\co \cap \Omega^\co}K(y-x)\de y\de x +\int_{E\cap\Omega^\co}\int_{E^\co\cap\Omega}K(y-x)\de y\de x,
	\end{split}\end{equation}
where $E$ and $\Omega$ are Lebesgue measurable sets in $\Rd$
and $K\colon \Rd \to [0,+\infty]$ is a Lebesgue measurable function
on which we prescribe suitable conditions.

In Section \ref{sec:over}, as a preliminary step,
we consider a general interaction functional between two sets:
	\[L_K(E,F)=\int_F\int_E K(y-x)\de y\de x\]
and we describe some of its basic properties.
We show that, 
when one of the sets of the couple has finite classical perimeter
the interaction is bounded by the $\mathrm{BV}$ norm of that set.
This, combined with the expression
of $\PerK(\,\cdot\,,\Omega)$ in terms of suitable couplings $L_K$,
comes in handy to prove that the functional in \eqref{eq:PerK}
is a perimeter according to Definition \ref{dfn:per}.

Once we know that \ref{eq:PerK} defines a perimeter,
in Subsection \ref{ssec:minsurf} we provide an existence result
for nonlocal minimal surfaces, i.e. sets that minimise $\PerK(\,\cdot\,,\Omega)$
among all the sets that coincide with a given one outside $\Omega$.
The proof of this takes into account the extension
of the perimeter functional to measurable functions that range in $[0,1]$
and it exploits the convexity of this extension;
in turn, convexity relies on the submodularity of $\Per(\,\cdot\,,\Omega)$
and on the validity of a generalised Coarea Formula.
Moreover, when the perimeter is built from a radial, strictly decreasing kernel
we are able to show that minimisers for Plateau's problem
with ``flat'' boundary conditions are halfspaces.

Section \ref{sec:Gamma} is devoted to a $\Gamma$-convergence argument.
We let $\Omega$ be an open bounded set with Lipschitz boundary
and we focus on the family of perimeter functionals $J_\epsilon(\,\cdot\,,\Omega)$
induced by mass preserving rescalings of a fixed kernel $K$, that is
	\begin{equation}\label{eq:rescK}
	K_\epsilon(h)\eqdef\frac{1}{\eps^d}K\left(\frac{h}{\eps}\right).
	\end{equation}
We are interested in the limiting behaviour of the ratios
$\frac{1}{\eps}J_\epsilon(\,\cdot\,,\Omega)$;
precisely, our intent is showing that
they $\Gamma$-converge w.r.t. the $L^1_\mathrm{loc}(\Rd)$-metric
to the classical perimeter in $\Omega$,
up to the multiplicative constant
	\begin{equation}\label{eq:cK}
	c_K \eqdef \frac{1}{2}\int_{\Rd} K(h)\abs{h_d}\de h
	\end{equation}
($h_d$ is the last component of the vector $h$).
Notice that the scaling factor $\frac{1}{\eps}$ is necessary to rule out trivial conclusions:
indeed, we have $\lim_{\eps \to 0} J_\epsilon(E,\Omega) =0$.
For the sake of completeness,
we recall the notion of $\Gamma$-convergence:
\begin{dfn}
	Let $(X,\de)$ be a metric space.
	The family $f_\epsilon\colon X\to[-\infty,+\infty]$
	$\Gamma$-converges w.r.t. the metric $\de$
	to the function $f_0\colon X\to [-\infty,+\infty]$ as $\epsilon\to0$ if
		\begin{enumerate}
			\item for any $x_0\in X$ and for any $\set{x_\eps}\subset X$ such that
					$x_\epsilon\to x_0$ it holds
					\[f_0(x_0)\leq\liminf_{\epsilon\to 0}f_\eps(x_\eps)\] 
			\item for any $x_0\in X$ there exists $\set{x_\eps}\subset X$ such that
					$x_\epsilon\to x_0$ and
					\[\limsup_{\epsilon\to 0}f_\eps(x_\eps)\leq f_0(x_0).\] 
		\end{enumerate}
\end{dfn}
The proofs of the inferior and of the superior limit inequality have very different natures:
in the first case (Subsection \ref{ssec:lli}),
we use a compactness criterion
to reduce the desired inequality to a density estimate,
while in the second (Subsection \ref{ssec:uli})
we give a pointwise convergence result
and then we conclude by a density lemma.
In spite of this diversity, there is a key point which is shared by the two arguments,
namely the possibility of controlling the rescaled interactions
	 \[\frac{1}{\eps}\int_F\int_E K_\epsilon(y-x)\de y \de x\]
in the limit $\eps\to0$: when $E$ and $F$ do not overlap,
Proposition \ref{stm:locdef} shows that asymptotically
these functionals either vanish or they are uniformly bounded,
depending on the mutual position of $E$ and $F$.

Before setting off the analysis,
we fix the notation we adopt throughout the paper
and we premise some reminders about the theory of finite perimeter sets.		
All the study is carried out
in the vector space $\Rd, d\geq 1,$ endowed with the Euclidean inner product $\,\cdot\,$
and the Euclidean norm $\abs{\,\cdot\,}\,$.
We shall often consider a reference set $\Omega \subset \Rd$,
assuming that it is open, connected and bounded.
When $\lambda>0$, $h\in\Rd$ and $E\subset\Rd$
we write $\lambda E +h$
to denote the set obtained from $E$
firstly by the dilation of factor $\lambda$ and then by the translation by $h$.
For any set $E\subset \Rd$, $E^\co$ is the complement of $E$ in $\Rd$
and $\chi_E$ denotes its characteristic function,
while $\abs{E}$ stands as its $d$-dimensional Lebesgue measure. We use the symbols 
$\call{L}^d$ and $\call{H}^{d-1}$ to denote respectively
the $d$-dimensional Lebesgue and the $(d-1)$-dimensional Hausdorff measure.
$\call{M}$ is the collection of all Lebesgue measurable sets in $\Rd$.
We shall systematically identify sets with their characteristic functions;
in particular, by saying that a sequence $\set{E_n}$ converges in $L^1_\mathrm{loc}(\Rd)$ to $E$
we mean that for any compact set $C$, the measure of the intersection   
$(E_n \symdif E)\cap C$ tends to $0$ as $n$ diverges.
There are two sets that play a distinguished role in what follows:
the halfspace
	\[H\coloneqq \set{x=(x_1,x_2,\dots,x_d)\in\Rd : x_d < 0 }\]
and the open unit cube
	\[U\coloneqq \left(-\frac{1}{2},\frac{1}{2}\right)^{\!d}.\]

If $u$ is a function, we use the symbols $\grad u$ and $\D u$
to denote respectively the classical and the distributional gradient of $u$;
in particular, $\D u$ is a $\Rd$-valued measure.
If $\Omega\subset\Rd$ is open,
we say that $u$ is a \emph{function of bounded variation}
when it belongs to $L^1(\Omega)$
and the total variation of the distributional gradient $\abs{\D u}$ is finite on $\Omega$:
	\[\mathrm{BV}(\Omega)\eqdef\set{u\in L^1(\Omega) : \abs{\D u}(\Omega) \mbox{ is finite}}.\]
This space can be characterised in terms of the $L^1$-norm of difference quotients:
	\begin{prop}\label{stm:charBV}
		Let $\Omega\subset \Rd$ be an open subset.
		Then, $u\colon\Omega \to \R$ is a function of bounded variation in $\Omega$ if and only if
		there exists a constant $c\geq 0$ such that
		for any $\Omega'$ compactly contained in $\Omega$ and
		for any $h\in\Rd$ with $\abs{h}<\mathrm{dist}(\Omega',\Omega^\co)$ it holds
			\[\norm{\tau_h u -u}_{L^1(\Omega')}\leq c\abs{h},\]
		where $\tau_h u(x)\eqdef u(x+h)$.
		In particular, it is possible to choose $c=\abs{\D u}(\Omega)$.
	\end{prop}
When $E\subset\Rd$ is a measurable set such that $\chi_E $ is a function of bounded variation
in a certain reference set $\Omega$, we say that $E$ has \emph{finite perimeter} in $\Omega$
or that it is a \emph{Caccioppoli set} in $\Omega$
and we put
	\begin{equation}\label{eq:DGper}
		\Per(E,\Omega)\eqdef\abs{\D \chi_E}(\Omega).
	\end{equation}
To recall the key result by De Giorgi and Federer,
let us consider for any $x\in \mathrm{supp}\abs{\D\chi_E}$ the Radon-Nikodym derivative
	\[\hat n(x)\coloneqq \der{\D\chi_E}{\abs{\D\chi_E}}(x) =
		\lim_{r\to 0^+}\frac{\D\chi_E(B(x,r))}{\abs{\D\chi_E}(B(x,r))};\]
here $B(x,r)$ is the open ball of centre $x$ and radius $r>0$.
We call the set
	\[\partial^\ast E\eqdef \set{x\in \Rd : \hat n(x) \mbox{ exists and has norm } 1}\]
the \emph{reduced boundary} of $E$ and when $x\in\partial^\ast E$
we also say that $\hat n(x)$ is the measure theoretic \emph{inner normal} to $E$ in $x$.
Now, the cited theorem states that if $E$ is a measurable set,
then $\partial^\ast E$ is $(d-1)$-rectifiable and 
$\D\chi_E=\hat n\chi_{\partial^\ast E}\call{H}^{d-1}$
so that 
	\begin{equation}\label{eq:per-Haus}
	\Per(E,\Omega)=\call{H}^{d-1}(\partial^\ast E\cap\Omega);
	\end{equation}
for this reason we shall call $\call{H}^{d-1}\llcorner \partial^\ast E$
the \emph{perimeter measure of} $E$.
In addition, for any $x\in\partial^\ast E$ there exists $R_x\in SO(d)$ such that
	\begin{equation}\label{eq:blowup}
	\frac{E-x}{r} \to R_x H \quad\mbox{in } L^1_\mathrm{loc}(\Rd) \mbox{ as } r\to 0^+.
	\end{equation}
We shall invoke these properties when proving the inferior limit inequality
of the $\Gamma$-convergence theorem.

For further details about the theory of functions of bounded variations and finite perimeter sets
we refer to the monographs by Ambrosio, Fusco and Pallara \cite{afp:functionsof} and by Maggi \cite{mag:setsof}.

\addcontentsline{toc}{subsection}{Acknowledgements}
\subsection*{Acknowledgements}
The authors gratefully thank Giovanni Alberti and Matteo Novaga for fruitful discussion.
Support for this work was provided by the Project PRA 2017-23 of the University of Pisa.
The stay of the first author in Pisa and of the second author in M\"unster
were funded respectively
by the DAAD by means of the Bundesministerium f\"ur Bildung und Forschung (BMBF)
and by ERC via Grant EU FP7 - ERC Consolidator Grant 615216 LifeInverse;
the authors thank the hosting universities for the kind hospitality.

\section{An overview of nonlocal perimeters}\label{sec:over}
	\subsection{The nonlocal perimeter associated with an integral kernel} Let $K\colon \Rd \to [0,+\infty]$ be a measurable function and
$E$, $F$ be sets in $\call{M}$;
we define the \emph{nonlocal $K$-interaction} between $E$ and $F$ as
	\[L_K(E,F)\coloneqq\int_F\int_E K(y-x)\de y\de x.\]
We can view this functional as the quantity of energy
that is stored in the couple of sets
because of the interaction expressed by the kernel $K$.

Notice that by Tonelli's Theorem
	\[L_K(E,F) = L_K(F,E) = \int_{E\times F}K(y-x)\de y \de x,\]
thus it is not restrictive to assume $K$ to be even:
	\begin{equation}\label{eq:Keven}
	K(h)=K(-h) \quad\mbox{for any } h\in\Rd.
	\end{equation}
The following facts can be derived in a straightforward manner:
	
	\begin{lemma}\label{stm:LK}
	Let $K\colon \Rd \to [0,+\infty]$ be a measurable function
	such that \eqref{eq:Keven} is satisfied. Then,
		\begin{enumerate}
		\item $L_K$ ranges in $[0,+\infty]$,
			and it vanishes if one its arguments has zero Lebesgue measure;
		\item for any $E_1,E_2,F\in\call{M}$, we have
				\[\begin{gathered}
				L_K(E_1,F) = L_K(E_2,F) \quad\mbox{if } \abs{E_1\symdif E_2}=0 \quad\mbox{and}\\
				L_K(E_1\cup E_2,F) = L_K(E_1,F)+L_K(E_2,F) \quad\mbox{if } \abs{E_1\cap E_2}=0;
				\end{gathered}\]
		\item for any $\lambda>0$ and $h\in\Rd$,
				\[L_K(\lambda E+h,F) = \lambda^{2d}\int_E\int_{\frac{1}{\lambda}(F-h)} K(\lambda(y-x))\de y\de x;\]
			in particular, $L_K$ is left unchanged if
			both arguments are translated by the same vector;
		\item the following equality holds:
			\begin{equation}\label{eq:LKvariant}
			L_K(E,F) = \int_{\Rd}K(h)\abs{E \cap (F-h)} \de h.
			\end{equation}
		\end{enumerate}	
	\end{lemma}

One may ask when the interaction $L_K$ is finite;
clearly, the answer heavily depends on the summability assumptions on $K$.
For instance, let us admit provisionally that $K$ is $L^1(\Rd)$;
then, from \eqref{eq:LKvariant} we see that
$L_K(E,F)$ is finite as soon as one of either $E$ or $F$ has finite Lebesgue measure
and it holds
	\begin{equation}\label{eq:estimateLK}
	L_K(E,F)\leq \norm{K}_{L^1(\Rd)}\min\set{\abs{E},\abs{F}}.
	\end{equation}
Further, assume that the support of $K$ is contained in a ball of radius $r$: we get
	\begin{equation*}\label{eq:nonloc}
	L_K(E,F)=\int_{\set{x\in E : \mathrm{dist}(x,F)<r}} \int_{\set{y\in F : \mathrm{dist}(y,E)<r}} K(y-x)\de y \de x,
	\end{equation*}
which shows that, for each set in the couple,
the points that play a major role are the ones that lie near to the other set;
similarly, in the general case, we expect that points that are separated by a large distance
have smaller influence on the total interaction $L_K$. 
We shall come back to this point later on,
when we consider the behaviour of functionals induced
by mass-preserving rescalings of $K$.  

From now on, we assume that	
	\begin{equation} \label{eq:summK}
	\begin{gathered}
	K\colon \Rd \to [0,+\infty) \mbox{ is a measurable even function such that} \\
	h\mapsto K(h)\min\set{1,\abs{h}} \mbox{ belongs to } L^1(\Rd).	
	\end{gathered}
	\end{equation}
If \eqref{eq:summK} is fulfilled, we are in position to prove that
the nonlocal interaction between a couple of sets is finite
provided we have some information on the mutual positions.
Indeed, if two sets overlap on a region of full Lebesgue measure,
we cannot expect $L_K$ to be finite
because $K$ might not be summable around the origin.

	\begin{prop}\label{stm:LKfinite}
		Let $E$ and $F$ be sets with strictly positive Lebesgue measure
		and let us assume that $E$ has finite perimeter in $\Rd$ and that $\vass{E\cap F}=0$.
		If \eqref{eq:summK} holds, then
			\begin{equation}
			L_K(E,F)\leq c(E)\int_{\Rd} K(h)\min\set{1,\abs{h}} \de h,
			\end{equation}
		where $c(E)\eqdef \max\set{\abs{E},\frac{\Per(E)}{2}}$.		
	\end{prop}
	\begin{proof}
		Up to Lebesgue negligible sets, $F\subset E^\co$; therefore,
			\[\begin{split}
			L_K(E,F) \leq & L_K(E,E^\co)=\frac{1}{2}\int_{\Rd}\int_{\Rd} K(h)\abs{\chi_E(x+h)-\chi_E(x)}\de x\de h \\
						= & \frac{1}{2}\int_{\set{\abs{h}<1}}K(h)\int_{\Rd}\abs{\chi_E(x+h)-\chi_E(x)}\de x\de h \\
						  &	+ \frac{1}{2}\int_{\set{\abs{h}\geq1}}K(h)\int_{\Rd}\abs{\chi_E(x+h)-\chi_E(x)}\de x\de h;
			\end{split}\]
		we estimate the last integral by the triangle inequality, while
		the assumption that $\chi_E$ is a function of bounded variation on $\Rd$
		provides the upper bound
			\[\int_{\Rd}\abs{\chi_E(x+h)-\chi_E(x)}\de x\leq\Per(E)\abs{h} \]
		(recall Proposition \ref{stm:charBV}) and hence, on the whole, we get
			\[	L_K(E,F) \leq \frac{1}{2}\Per(E)\int_{\set{\abs{h}<1}}K(h)\abs{h}\de h
						  + \abs{E}\int_{\set{\abs{h}\geq1}}K(h)\de h. \]			
	\end{proof}
Now, we use the functional $L_K$
to recall the definition of nonlocal perimeter.
We firstly fix a reference set $\Omega\in\call{M}$
and, to avoid trivialities,
hereafter we always assume that it has strictly positive measure.
Let us define the \emph{nonlocal perimeter} of a set $E\in\call{M}$ in $\Omega$:
	\begin{equation}\label{eq:defPerK-O}
	\begin{split}
	\PerK(E,\Omega) \eqdef& L_K(E\cap\Omega,E^\co\cap\Omega) \\
						& + L_K(E\cap\Omega,E^\co \cap \Omega^\co) + L_K(E\cap\Omega^\co,E^\co\cap\Omega);
	\end{split}
	\end{equation}
as a particular case, we set
	\begin{equation}\label{eq:defPerK-Rd}
	\PerK(E)\coloneqq\PerK(E,\Rd) = L_K(E,E^\co)
	\end{equation}
and we observe that $\PerK(E,\Omega)=L_K(E,E^\co)=\PerK(E)$ whenever $\vass{E\cap\Omega}=0$.
These positions rely on the intuitive notion of perimeter
that we discussed in the introductory section:
we attempt to identify the locus that divides a set $E$ from its complement
and we do this by considering suitable $K$-couplings between $E$ and $E^\co$.
On one hand, this is evident from Definition \eqref{eq:defPerK-Rd},
on the other this is true for Definition \eqref{eq:defPerK-O} as well,
the only difference being 
the omission of the interactions that arise inside $\Omega^\co$.
More precisely, one can understand the nonlocal perimeter of a set $E$ in $\Omega$
as being made of two contributions:
the former is expressed by the summand $L_K(E\cap\Omega,E^\co\cap\Omega)$
and it encodes the energy that is located in $\Omega$,
while the latter is provided by $L_K(E\cap\Omega,E^\co \cap \Omega^\co) + L_K(E\cap\Omega^\co,E^\co\cap\Omega)$
and it captures the energy that ``flows''
through the portions of the boundaries that $E$ and $\Omega$ share.
When treating the $\Gamma$-convergence of the perimeter,
we shall see that these different natures give birth to distinct asymptotics.

We gather here some examples of kernels that fulfil the assumptions in \eqref{eq:summK}

\begin{exs}
	Of course, perimeters associated to $L^1$ kernels fit into our theory.
	Outside this class, a relevant example is given by fractional kernels
	(\cite{crs:nonlocal,lud:anisotropicfractional}), that is
	\[K(h)=\frac{a(h)}{\abs{h}^{d+s}},\]
	where $s\in(0,1)$ and $a\colon \Rd \to \R$ is a measurable even function
	such that $0<m\leq a(h)\leq M$ for any $h\in\Rd$ for some positive $m$ and $M$.
	A third case is represented by the kernels
	we shall deal with most of the times in the sequel,
	namely the functions $K\colon \Rd \to [0,+\infty)$
	such that the map $h\mapsto K(h)\abs{h}$ is $L^1$;
	observe that this summability assumption allows for a fractional-type behaviour near the origin,
	but it also implies faster-than-$L^1$ decay at infinity. 
\end{exs}

By now, the literature concerning nonlocal-perimeter-like functionals is expanding.
For instance, the mentioned class fractional perimeters, i.e.
	\[\begin{split}
		\Per_s(E,\Omega) \eqdef & \int_{E\cap\Omega}\int_{E^\co\cap\Omega}\frac{\de y\de x}{\abs{y-x}^{d+s}} \\
								& +\int_{E\cap\Omega}\int_{E^\co\cap\Omega^\co}\frac{\de y\de x}{\abs{y-x}^{d+s}}
								+ \int_{E\cap\Omega^\co}\int_{E^\co\cap\Omega}\frac{\de y\de x}{\abs{y-x}^{d+s}}
	\end{split}\]
has been extensively studied; here, we wish to mention just
\cite{crs:nonlocal}, where existence and regularity of solutions to Plateau's problem are dealt with,
and the papers \cite{cv:uniformestimates} by Caffarelli and Valdinoci
and \cite{adm:gammaconvergence} by Ambrosio, De Philippis and Martinazzi,
where the limiting behaviour  as $s\to1^-$
of $\Per_s(\,\cdot\,,\Omega)$ and of the related minimal surfaces are discussed.
The analysis for general kernels $K$ has been carried out in several directions as well
and, as a short selection of known results, we cite
the flatness properties for minimal surfaces in \cite{csv:quantitativeflatness},
the existence of isoperimetric profiles established by Cesaroni and Novaga in \cite{cn:theisoperimetric}
and the study of nonlocal curvatures by Maz{\'o}n, Rossi and Toledo in \cite{mrt:nonlocalperimeter}.

We close this Subsection by proving
that the functional $\PerK$ is a perimeter
in the axiomatic sense introduced in \cite{cmp:nonlocalcurvature}.
Starting from the properties of $L_K$ that are shown in Lemma \ref{stm:LK},
it is easy to check that
statements \ref{stm:per1}, \ref{stm:per2} and \ref{stm:per3}
in Definition \ref{dfn:per} hold true;
in addition, once one has observed that
	\begin{equation}\label{eq:semicont}
	\begin{split}
	\PerK(E,\Omega)=&\,\frac{1}{2}\int_\Omega\int_\Omega K(y-x)\abs{\chi_E(y)-\chi_E(x)}\de y\de x \\
					&+ \int_{\Omega}\int_{\Omega^\co}K(y-x)\abs{\chi_E(y)-\chi_E(x)}\de y \de x,
	\end{split}
	\end{equation}
semicontinuity \ref{stm:per5} follows by Fatou's Lemma.
To prove submodularity \ref{stm:per6} 
it suffices to decompose the involved sets in a suitable manner:
for instance, one can find
	\[\begin{split}
	L_K((E\cup F)\cap \Omega,&\,E^\co \cap F^\co\cap\Omega) \\
		=\,&\, L_K(E\cap\Omega,E^\co \cap\Omega)+L_K(F\cap\Omega,F^\co \cap\Omega) \\
		&\, -L_K(E\cap\Omega,E^\co \cap F\cap\Omega)-L_K(F\cap\Omega,E\cap F^\co \cap\Omega) \\
		&\, - L_K(E\cap F\cap \Omega,E^\co\cap F^\co\cap \Omega)
	\end{split}\]
and
	\[\begin{split}
	L_K(E\cap F\cap \Omega,\,&\,(E^\co \cup F^\co)\cap\Omega) \\
		=\,&\, L_K(E\cap F\cap \Omega,E^\co\cap F^\co\cap \Omega) \\
			&\,+L_K(E\cap F\cap \Omega,E^\co \cap F\cap\Omega)+L_K(E\cap F\cap \Omega,E\cap F^\co \cap\Omega),
	\end{split}\] 
so that on the whole one gets
	\[\begin{split}
	\PerK(E,\Omega)+&\PerK(F,\Omega) \\
		=\,&\, \PerK(E\cap F,\Omega)+\PerK(E\cup F,\Omega)   \\
		&\, +2 L_K(E\cap F^\co\cap \Omega,E^\co\cap F\cap \Omega)
			+2L_K(E\cap F^\co\cap \Omega,E^\co\cap F\cap \Omega^\co) \\
		&\,	+2L_K(E\cap F^\co\cap \Omega^\co,E^\co\cap F\cap \Omega).
	\end{split}\]
Eventually, we are left to show that also \ref{stm:per4} is satisfied.

\begin{prop}\label{stm:valdinoci}
	Let us assume that \eqref{eq:summK} holds
	and suppose that $\Omega$ is an open set with finite Lebesgue measure.
	Then, if $E$ is a Caccioppoli set in $\Rd$,
		\[\PerK(E,\Omega)\leq c(E,\Omega)\int_{\Rd} K(h)\min\set{1,\abs{h}} \de h,\]
	where $c(E)\eqdef \max\set{\frac{\Per(E)}{2},\abs{\Omega}}$.
	In particular,$E$ has finite nonlocal $K$-perimeter in $\Omega$ as well and
	$\PerK(\,\cdot\,,\Omega)$ is a perimeter
	in the sense of Definition \ref{dfn:per}.
\end{prop}
\begin{proof}
	The conclusion can be obtained imitating the proof of Proposition \ref{stm:LKfinite};
	see also Proposition \ref{stm:JK-BV}.
\end{proof}
	\subsection{Extension to functions and nonlocal minimal surfaces}\label{ssec:minsurf} Of course one is led to consider the perimeter
as a geometric property attached to a set;
nevertheless, we know that
the classic notion by De Giorgi can be casted
in the framework of functions of bounded variation.
Here, we present a construction of the same flavour,
whose aim is extending the functional $\PerK$ to functions.
This can be achieved in a natural way: 
grounding on identity \eqref{eq:semicont},
we are induced to set for any measurable $u\colon \Rd \to \R$
	\begin{equation}\begin{split}
	J_K^1(u,\Omega) &\coloneqq\int_\Omega\int_\Omega K(y-x)\abs{u(y)-u(x)}\de x\de y, \\
	J_K^2(u,\Omega) &\coloneqq\int_{\Omega}\int_{\Omega^\co}K(y-x)\abs{u(y)-u(x)}\de x \de y \quad\mbox{and} \\
	J_K(u,\Omega) &\coloneqq \frac{1}{2}J_K^1(u,\Omega) + J_K^2(u,\Omega).
	\end{split}
	\end{equation}
We shall refer to $J_K(\,\cdot\,,\Omega)$ as \emph{nonlocal $K$-energy functional}
and it can be easily seen that it is lower semicontinuous
w.r.t. $L^1_\mathrm{loc}(\Rd)$-converge.
By a small abuse of notation,
we shall write $J_K^i(E,\Omega)$ for $i=1,2$ and $J_K(E,\Omega)$ when the functionals are evaluated
on the characteristic function of $E$, so that
\begin{equation*}\begin{split}
\frac{1}{2}J_K^1(E,\Omega) &= L_K(E\cap\Omega,E^\co\cap\Omega) \\
J_K^2(E,\Omega) &= L_K(E\cap\Omega,E^\co \cap \Omega^\co)
+ L_K(E^\co\cap\Omega,E\cap\Omega^\co)\quad\mbox{and} \\
\PerK(E,\Omega) &= J_K(E,\Omega) =\frac{1}{2}J_K^1(E,\Omega) + J_K^2(E,\Omega).
\end{split}
\end{equation*}
In view of these equalities, we shall informally say that
the functional $J^1_K(\,\cdot\,,\Omega)$ is the local contribution to the perimeter,
while $J^2_K(\,\cdot\,,\Omega)$ is the nonlocal one.

In the previous subsection,
we gave some heuristic justification to the definition of nonlocal perimeter
and then we also proved that this object owns certain ``reasonable'' properties;
amongst them, there is the finiteness of the $K$-perimeter for regular sets.
Actually, if we suppose that
\begin{description}
	\item[C1]\label{stm:C1} $\Omega$ is an open, connected and bounded subset of $\Rd$
			with Lipschitz boundary and that
	\item[C2]\label{stm:C2} $K\colon \Rd \to [0,+\infty)$ is a measurable
			even function such that the quantity
	\begin{equation}\label{eq:c'K}
	c'_K \eqdef \int_{\Rd} K(h)\abs{h}\de h \quad\mbox{is finite,}
	\end{equation}
\end{description}
then not only the theory we have developed so far applies,
but we can also prove a broader result involving functions of bounded variation
that yields a conclusion
which is similar in spirit to the one of  Proposition \ref{stm:valdinoci}.

\begin{prop}\label{stm:JK-BV}
	Let us assume that conditions \textbf{C1} and \textbf{C2} are fulfilled.
		\begin{enumerate}
			\item If $\Omega$ is convex and $u\in\mathrm{BV}(\Omega)$,
				then
					\begin{equation}\label{eq:JK1-BV}
					J_K^1(u,\Omega) \leq c'_K \abs{\D u}(\Omega).
					\end{equation}
			\item If $u\in C^1(\Rd)\cap\mathrm{BV}(\Rd)$, then
					\begin{equation}\label{eq:JK-BV}
					J_K(u,\Omega) \leq c'_K \int_{\Rd}\abs{\nabla u}.
					\end{equation}
			\item If $u\in\mathrm{BV}(\Rd)$, \eqref{eq:JK-BV} holds as well,
				on condition that one replaces 
				the integral on the right-hand side with $\abs{\D u}(\Rd)$.
		\end{enumerate}
 
\end{prop}
\begin{proof}
	Let us firstly assume that $\Omega$ is a convex open subset in $\Rd$
	and that $u\in\mathrm{BV}(\Omega)$.
	By the change of variables $h=y-x$ we find
		\[ J_K^1(u,\Omega)=\int_{\Rd}K(h)\int_{\{x\in\Omega : x+h\in\Omega\}} \abs{u(x+h)-u(x)}\de x\de h \]
	and, subsequently, by the characterisation of $\mathrm{BV}$ functions
	recalled in Proposition \ref{stm:charBV},
		\[ J_K^1(u,\Omega)\leq \left(\int_{\Rd}K(h)\abs{h}\de h\right) \abs{\D u}(\Omega),\]
	that is \eqref{eq:JK1-BV}.
	
	Next, suppose that $u\in C^1(\Rd)\cap\mathrm{BV}(\Rd)$;
	similarly to the previous lines, we infer
		\[J_K(u,\Omega) \leq \int_{\Rd}K(h)\int_\Omega \abs{u(x+h)-u(x)}\de x\de h,\]
	but under the current hypotheses we can no longer localise the points
	of the segment from $x$ to $x+h$ and thus we integrate over the whole space:
		\[J_K(u,\Omega) \leq \int_{\Rd}K(h)\abs{h}\de h \int_{\Rd} \abs{\nabla u(\xi)}\de \xi.\]
		
	Finally, recall that if $u\in\mathrm{BV}(\Rd)$, then
	there exist a sequence $\set{u_n}\subset C^\infty(\Rd)\cap\mathrm{BV}(\Rd)$
	that converges to $u$ in $L^1(\Rd)$ and that satisfies
		\[\limsucc{n}\int_{\Rd}\abs{\nabla u_n} = \abs{\D u}(\Rd).\]
	Also, thanks to $L^1_\mathrm{loc}(\Rd)$-lower semicontinuity,
	we deduce the last statement from the second by an approximation argument.
\end{proof}	

The functional $J_K$ and the $K$-perimeter are further linked by a coarea-type result
(see \cite{csv:quantitativeflatness} and also
\cite{adm:gammaconvergence,cn:theisoperimetric,,mrt:nonlocalperimeter} for analogous statements).

\begin{prop}[Coarea formula]\label{stm:coarea}
	If $K\colon \Rd \to [0,+\infty)$ is measurable,
	then for any measurable function $u\colon \Rd \to [0,1]$
		\[ J_K^1(u,\Omega) = \int_0^1 J_K^1(\set{u>t},\Omega)\de t \quad\mbox{and}\quad
			J_K^2(u,\Omega) = \int_0^1 J_K^2(\set{u>t},\Omega)\de t
		\]
	and hence
	\[ J_K(u,\Omega) = \int_0^1\PerK(\set{u>t},\Omega)\de t \]
\end{prop}
\begin{proof}
	Given $x,y\in\Omega$, let us suppose without loss of generality that $u(x)\leq u(y)$;
	we consider the function $[0,1]\ni t\mapsto\chi_{\set{u>t}}(x)-\chi_{\set{u>t}}(y)$
	and we notice that it is different from $0$ exactly when $t\in[u(x),u(y)]$.
	Consequently,
		\[\abs{u(x)-u(y)}=\int_{0}^1 \abs{\chi_{\set{u>t}}(x)-\chi_{\set{u>t}}(y)}\de t\]
	and, by Tonelli's Theorem,
	\[\begin{split}
	J^1_K(u,\Omega) \eqdef & \int_{\Omega}\int_{\Omega}K(x-y)\abs{u(x)-u(y)}\de x\de y \\
					= & \int_0^1\int_{\Omega}\int_{\Omega} K(x-y)\abs{\chi_{\set{u>t}}(x)-\chi_{\set{u>t}}(y)}\de x\de y\de t \\
					= & \int_{0}^{1}J_K^1(\set{u>t},\Omega)\de t
	\end{split}\]
	In a similar way, one also proves that
	the equality concerning $J_K^2(u,\Omega)$ holds.
\end{proof}

The validity of coarea formula is crucial for variational purposes.
Indeed, it allows to invoke two abstract results proved in \cite{cgl:continuouslimits}
by Chambolle, Giacomini and Lussardi:

	\begin{teor}
		If $J\colon L^1(\Omega)\to [0,+\infty]$ is
		a proper lower semicontinuous functional such that
			\begin{equation}\label{eq:gencoarea}
			J(u)=\int_{-\infty}^{+\infty}J(\chi_{\set{u>t}})\de t
			\end{equation}
		and that
			\[J(\chi_{E\cap F}) + J(\chi_{E\cup F}) \leq J(\chi_{E}) + J(\chi_{F})\]
		for any couple of measurable sets in $\Omega$,
		then $J$ is convex.
	\end{teor}
	
	\begin{teor}\label{stm:cgl}
		Let $\set{J_n}_{n\in\N}$ be a sequence of convex functionals
		such that \eqref{eq:gencoarea} holds and
		let us suppose that there exists a functional
		$\tilde J$ defined on measurable sets of $\Omega$
		such that the sequence obtained by restriction
		of the functionals $J_n$ to measurable sets 
		$\Gamma$-converges to $\tilde J$ w.r.t. the $L^1$-convergence.
		Then, the sequence $\set{J_n}$ $\Gamma$-converges to $J$ w.r.t. the same norm if we put
			\[J(u)=\int_{-\infty}^{+\infty}\tilde J(\chi_{\set{u>t}})\de t.\]		
	\end{teor}
	
The latter of the two theorems above is relevant
for the discussion contained in Section \ref{sec:Gamma}
concerning the limiting properties of nonlocal perimeters.
For the moment being, we take advantage of the former and we infer

\begin{cor}
	If $K\colon \Rd \to [0,+\infty)$ is measurable,
	the functional $J_K(\,\cdot\,,\Omega)$ is convex on $L^1(\Rd;[0,1])$.
\end{cor}

At this stage, we are in position to
solve a Plateau-type problem for nonlocal perimeters
through the direct method of calculus of variations.
Notice that strong convergence of minimising sequences in $L^1$ is not guaranteed in principle,
because a uniform bound on the nonlocal perimeter is very weak information;
for example, if the kernel $K$ is $L^1$ and $\Omega$ is bounded,
then any measurable $E$ satisfies $\PerK(E,\Omega)\leq3\norm{K}_{L^1(\Rd)}\abs{\Omega}$.
We circumvent this obstacle by making use of convexity,
which permits to draw the conclusion from weak compactness only.

\begin{teor}[Existence of solutions to Plateau's problem]\label{stm:plateau}
	Let $K\colon \Rd \to [0,+\infty)$ be measurable
	and let $\Omega\subset \Rd$ be open and bounded.
	Suppose that $E_0\in\call{M}$ has finite $K$-perimeter in $\Omega$ and define
		\[\call{F}\coloneqq\set{F\in\call{M} : \PerK(F,\Omega)<+\infty \mbox{ and } F \cap\Omega^\co = E_0\cap\Omega^\co}.\]
	Then, there exists $E\in\call{F}$ such that
		\[\PerK(E,\Omega)\leq\PerK(F,\Omega) \quad\mbox{for any } F\in\call{F}.\]
	Also, any minimiser satisfies
	\begin{align}
	L_K(E,F)&\leq L_K(E^\co\cap F^\co,F) \quad \mbox{whenever } F\subset E^\co\cap\Omega \quad\mbox{and} \label{eq:optcond1}\\
	L_K(E^\co,F)&\leq L_K(E\cap F^\co,F) \quad \mbox{whenever } F\subset E\cap\Omega \label{eq:optcond2}
	\end{align}
\end{teor}
\begin{proof}	
	Let us consider a minimising sequence $\set{u_n}$
	for the more general minimisation problem 
		\[\inf\set{J_K(v,\Omega) : v\colon\Rd\to[0,1], v\mbox{ measurable},
					J_K(v,\Omega)<+\infty \mbox{ and } v\rvert_{\Omega^\co}=u_0},\]
	where $u_0\coloneqq\chi_{E_0\cap\Omega^\co}$;
	notice that the set of competitors is non-empty,
	because it contains at least $\chi_{E_0}$.
	We also observe that, for any choice of $p\in(1,+\infty)$,
	$\set{u_n}$ is bounded in $L^p(\Omega;[0,1])$ and
	therefore there exists $u\in L^p(\Omega;[0,1])$
	such that $u_n\rvert_\Omega$ weakly converges to it, up to subsequences.
	We extend $u$ outside $\Omega$ setting $u\rvert_{\Omega^\co} = u_0$
	and with this choice we get
		\[\lim_{\diverge{n}}J_K(u_n,\Omega)\geq J_K(u,\Omega).\]
	Indeed,	$J_K$ is convex and lower semicontinuous w.r.t. strong convergence in $L^1(\Rd;[0,1])$
	and hence it is also weakly lower semicontinuous in $L^p(\Omega;[0,1])$ for any $p\in[1,+\infty)$,
	which implies immediately $\liminf_{\diverge{n}}J^1_K(u_n,\Omega)\geq J^1_K(u,\Omega)$;
	the analogous inequality for the nonlocal term follows as well
	noticing that $u_n = u = u_0$ in $\Omega^\co$. 
	Hence, $u$ is a minimiser for $J_K(\,\cdot\,,\Omega)$.
		
	At this stage, the statement concerning existence is proved
	once we show that from any function that minimises $J_K(\,\cdot\,,\Omega)$
	one can recover a set $E$ that minimises $\PerK(\,\cdot\,,\Omega)$.
	To this purpose, we apply the Coarea formula: given that
		\[J_K(u,\Omega) = \int_0^1\PerK(\set{u>t},\Omega)\de t,\]
	for some $t^*\in(0,1)$ it must hold
	$J_K(u,\Omega) \geq \PerK(\set{u>t^*},\Omega)$;
	then, just set $E=\chi_{\set{u>t^*}}$.
	
	Eventually, we prove inequalities \eqref{eq:optcond1} and \eqref{eq:optcond2}.
	Suppose that $E$ minimises the perimeter and that $F\subset E^\co \cap \Omega$;
	then, the inequality $\PerK(E,\Omega)\leq\PerK(E\cup F,\Omega)$ holds
	and we rewrite it as
		\[L_K(E\cap\Omega,E^\co)+L_K(E\cap\Omega^\co,E^\co\cap\Omega) \leq 
			L_K((E\cup F)\cap\Omega,E^\co\cap F^\co)+L_K(E\cap\Omega^\co,E^\co\cap F^\co\cap\Omega).\]
	We decompose the first term in the left-hand side
	and we confront the second summands on each side, getting
		\[L_K(E\cap\Omega,F)+L_K(E\cap\Omega,E^\co\cap F^\co)+L_K(E\cap\Omega^\co,F) \leq L_K((E\cup F)\cap\Omega,E^\co\cap F^\co) \]
	and therefore we find
		\[L_K(E\cap\Omega,F)+L_K(E\cap\Omega^\co,F) \leq L_K(F,E^\co\cap F^\co), \]
	which is \eqref{eq:optcond1}.
	The other inequality can be proved similarly starting from
	$\PerK(E,\Omega)\leq\PerK(E\cap F^\co,\Omega)$.
\end{proof}

We borrowed the proof of optimality conditions from \cite{adm:gammaconvergence,crs:nonlocal},
where analogous results are stated for fractional perimeters;
notice that to the validity of \eqref{eq:optcond1} and \eqref{eq:optcond2}
no restriction on $K$ is needed.
On the contrary, to deduce some extra information on minimisers,
still following the same papers,we shall require that
	\begin{description}
		\item[C3] $\bar K\colon [0,+\infty) \to [0,+\infty)$ is a measurable function
			and for any $h\in\Rd$, $K(h)\eqdef\bar K(r)$ if $\abs{h}=r$;
		\item[C4] $\bar K$ is strictly decreasing.
	\end{description}
When $K$ satisfies condition $\textbf{C3}$,
the coupling $L_K$ is left unchanged by isometries:
	\begin{equation}\label{eq:rotinv}
		L_K(R(E),R(F))=\int_E\int_F K(R(y-x))\de y\de x = L_K(E,F) \quad \mbox{for any isometry } R
	\end{equation}

\begin{prop}[Flatness of minimisers]\label{stm:flatness} 
	Let us assume that \textbf{C3} and \textbf{C4} hold and let $E\in\call{M}$.
	\begin{enumerate}
		\item\label{stm:flat1} If \eqref{eq:optcond1} holds for $E$ with $\Omega = U$ and
		$H\cap U^\co\subset E$, then
		$H\subset E$ up to a set of measure zero.
		\item\label{stm:flat2} If \eqref{eq:optcond2} holds for $E$ with $\Omega = U$ and
		$E\cap U^\co\subset H$, then
		$E\subset H$ up to a set of measure zero.
	\end{enumerate}
	Also, the same statements hold true replacing $H$ by $H^\co$
	and if $E$ is a minimiser for the problem
	$\inf\set{\PerK(F,U): F\cap U^\co=H\cap U^\co}$,
	then, $\abs{E\symdif H}=0$.
\end{prop}
\begin{proof}
	Let us provisionally assume that \ref{stm:flat1} holds both for $H$ and $H^\co$;
	then \ref{stm:flat2} follows. Indeed, if $E$ fulfils \eqref{eq:optcond2},
	then $E^\co$ satisfies \eqref{eq:optcond1} and hence,
	by applying \ref{stm:flat1} with $H^\co$,
	we get $\abs{H^\co\cap U \cap E}=0$, as desired.
	
	Consequently, if $E$ is a solution to Plateau's problem
	\[\inf\set{\PerK(F,U): F\cap U^\co=H\cap U^\co},\]
	by Theorem \ref{stm:plateau}, \eqref{eq:optcond1} and \eqref{eq:optcond2} hold
	and thanks to the constraint $E\cap U^\co=H\cap U^\co$
	we can invoke both \ref{stm:flat1} and \ref{stm:flat2},
	thus concluding $\abs{E\symdif H}=0$.
	
	Finally, we turn to the proof of \ref{stm:flat1}.
	The idea is to apply \eqref{eq:optcond1} with a suitable competitor.
	Since we suppose $H\cap U^\co \subset E$, $F^-\eqdef H\cap E^\co$ is contained in $Q$.
	Let us put $F^+\eqdef R(F^-)\cap E^\co$ and $F\eqdef F^-\cup F^+$,
	where $R(x_1,\dots,x_{d-1},x_d)=(x_1,\dots,x_{d-1},-x_d)$.
	Notice that $F\subset E^\co \cap U$ and hence,
	taking advantage of \eqref{eq:rotinv}, we have
		\[L_K(E,F)\leq L_K(E^\co\cap F^\co, F) = L_K(G,R(F)),\]
	where $G\eqdef R(E^\co\cap F^\co)$.
	$F$ can be decomposed as the disjoint union of
	$F'\eqdef F^-\setminus R(F^+)$ and $F''\eqdef F^+\cup R(F^+)$,
	so that the inequality above becomes
		\[\begin{split}
			L_K(E,F)& \leq L_K(G,R(F'))) + L_K(G,F'') \\
					& = L_K(G,R(F')) - L_K(G,F') + L_K(G,F),
		\end{split}\]
	that is
		\[L_K(E,F) - L_K(G,F)\leq L_K(G,R(F')) - L_K(G,F').\]
	We observe that
	the left-hand side can be rewritten as $L_K(E\cap G^\co,F)$, yielding
		\[0\leq L_K(E\cap G^\co,F)\leq L_K(G,R(F')) - L_K(G,F').\]
	Nevertheless, if $F'$ is not negligible,
	the last quantity is always strictly negative because,
	for any $x\in G$, $\abs{y-x} < \abs{R(y)-x}$ if $y\in F'\cap\set{y : y_d\neq 0}$
	and $K$ is a radially strictly decreasing function;
	it follows that $\abs{F'}=0$ and either $\abs{E\cap G^\co}=0$ or $\abs{F}=0$.
	The latter of these conditions immediately implies the conclusion
	since $H\cap E^\co=F^-\subset F$.
	
	Let us assume instead that $\abs{E\cap G^\co}=0$.
	We repeat the argument that we have just outlined above
	to a perturbation of $E$; namely, for any $\eps >0$, we set $E_\eps \eqdef E + (0,0,\dots, \eps)$
	and we observe that $E_\eps$ satisfies $\eqref{eq:optcond1}$ with
	$\Omega = Q_\epsilon\eqdef Q+(0,0,\dots,\eps)$
	and thus also with $\Omega = \tilde Q_\eps \eqdef Q_\epsilon\cap R(Q_\eps)$.
	We next define $F^-_\epsilon, F^+_\epsilon, F_\epsilon, F'_\eps$ and $F''_\eps$
	in complete analogy with the sets $F^-, F^+, F, F'$ and $F''$ introduced in the previous lines
	and we infer that $\abs{F'_\eps}=0$ and either $\abs{E_\eps\cap G^\co_\eps}=0$ or $\abs{F_\eps}=0$.
	The point is that now it holds $\abs{E_\eps\cap G^\co_\eps}=\infty$,
	thus $F_\epsilon$ in necessarily negligible and $\abs{H\cap E_\eps}=0$;
	finally, let $\epsilon$ tend to $0$.	
	
	Similarly to the lines above,
	one can prove that the conclusions are not compromised if $H$ is replaced by $H^\co$
	and in this way the proof is concluded.
\end{proof}

We shall exploit the flatness result for minimisers
to prove a useful characterisation of the constant $c_K$
appearing in Theorem \ref{stm:Gammaconv},
similarly to what is done in \cite{adm:gammaconvergence}.

\section{$\Gamma$-convergence of nonlocal perimeters}\label{sec:Gamma}
	In this section we turn to a $\Gamma$-convergence result
of mass preserving rescalings of the $K$-perimeter.	
Hereafter we assume that \textbf{C1} and \textbf{C3} hold.
Let us suppose in addition that
	\begin{description}
		\item[C2'] the quantity \[\int_0^{+\infty} \bar K(r)r^d \de r \quad\mbox{is finite.}\]
	\end{description}
The combination of \textbf{C2'} and \textbf{C3} guarantees that \textbf{C2} holds as well:
indeed, the implication is trivial when $d=1$,
while when $d\geq 2$
	\[c'_K=\int_{\Rd} K(h)\abs{h}\de h = \int_0^{+\infty}\int_{\partial B(0,r)} \bar K(r)r \de\call{H}^{d-1}(z)\de r
		= d\omega_d \int_{0}^{+\infty}\bar K(r)r^d\de r.\] 
Besides, thanks to radial symmetry, if $d\geq 2$,
we have the following chain of equalities: 
	\[\begin{split}
		\int_{\Rd} K(h)\abs{h_d}\de h &=
				\int_0^{+\infty}\int_{\partial B(0,r)} \bar K(r)\abs{e_d\cdot z}\de\call{H}^{d-1}(z)\de r \\
				& = \int_{\partial B(0,1)}\abs{e_d\cdot z}\de\call{H}^{d-1}(z) \int_{0}^{+\infty} \bar K(r)r^d\de r \\
				& = \frac{\int_{\partial B(0,1)}\abs{e_d\cdot z}\de\call{H}^{d-1}(z)}{d\omega_d}\int_{\Rd} K(h)\abs{h}\de h;
	\end{split}\]
thus, recalling \eqref{eq:cK}, we have
	\begin{equation}\label{eq:h=hd}
	c_K = \frac{\alpha_{1,d}}{2} c'_K, \quad\mbox{with } 
	\alpha_{1,d} \eqdef \dfrac{\int_{\partial B(0,1)}\abs{e_d\cdot z}\de\call{H}^{d-1}(z)}{d\omega_d}.
	\end{equation}
Summing up, if we assume the validity of
\textbf{C1}, \textbf{C2'} and \textbf{C3},
then the theory of Section \ref{sec:over} applies,
the only exception being Proposition \ref{stm:flatness},
which also requires \textbf{C4}.

In view of the forthcoming analysis, 
it is convenient to fix some further notation.
For $\epsilon>0$ and $h\in\Rd$ recall position \eqref{eq:rescK}
and for $E,F\in\call{M}$ let us define the functionals
	\[\begin{gathered}
		L_\epsilon(E,F)\eqdef L_{K_\eps}(E,F), \\
		J_\epsilon^1(E,\Omega) \coloneqq J_{K_\eps}^1(E,\Omega), \quad
		J_\epsilon^2(E,\Omega) \coloneqq J_{K_\eps}^2(E,\Omega) \quad\mbox{and} \\
		J_\epsilon(E,\Omega) \coloneqq \frac{1}{2}J_\eps^1(E,\Omega) + J_\epsilon(E,\Omega).
	\end{gathered}\]

Our main goal is proving the following result:

\begin{thm}\label{stm:Gammaconv}
	Let us suppose that \textbf{C1}, \textbf{C2'} and \textbf{C3} are fulfilled
	and let $E\in\call{M}$;
	then,
	\begin{enumerate}
		\item\label{stm:Gammasup}there exist a family $\set{E_\epsilon}_{\epsilon>0}$
		that converges to $E$ in $L^1_\mathrm{loc}(\Rd)$ with the property that
		\[\limsup_{\epsilon\to0} \frac{1}{\eps}J_\epsilon(E_\eps,\Omega)\leq c_K\Per(E,\Omega);\]
		\item\label{stm:Gammainf} if \textbf{C4} holds too,
		for any family $\set{E_\epsilon}_{\epsilon>0}$
		that converges to $E$ in $L^1_\mathrm{loc}(\Rd)$,
		\[c_K\Per(E,\Omega)\leq \liminf_{\epsilon\to0} \frac{1}{\eps}J_\epsilon^1(E_\eps,\Omega).\]	
	\end{enumerate}
\end{thm}

The functionals $J_\epsilon^2(\,\cdot\,,\Omega)$ are positive
and thus, evidently, the Theorem above implies
the $\Gamma$-converge of the ratios $\frac{1}{\eps}J_\epsilon(\,\cdot\,,\Omega)$
to $c_K\Per(\,\cdot\,,\Omega)$ w.r.t. the $L^1_\mathrm{loc}(\Rd)$-distance.
The two contributions $J^1_\epsilon$ and $J^2_\epsilon$
that compound the rescaled perimeter functional $J_\epsilon$ play different roles:
qualitatively, when $\epsilon$ is small,
the former is concentrated near the portions of the boundary of $E$ inside $\Omega$,
the latter instead gathers around the portions that are close to the boundary of $\Omega$;
this is made precise by Proposition\ref{stm:pointconv},
which shows that the pointwise limit and the $\Gamma$-limit do not agree in general.

The analogous of Theorem \ref{stm:Gammaconv} for the case of fractional perimeters was established
by Ambrosio, De Philippis and Martinazzi in \cite{adm:gammaconvergence};
notice that, however, the scaling used in that work is different,
even if we can still adopt similar techniques.
In particular, following \cite{adm:gammaconvergence} and the work \cite{ab:anonlocal} by Alberti and Bellettini
concerned with anisotropic phase transitions,
we prove the lower limit inequality
via the strategy introduced by Fonseca and M\"uller in \cite{fm:relaxationof},
which amounts to turn the proof of \ref{stm:Gammainf}
into an inequality of Radon-Nikodym derivatives. 

On the other hand, proofs of upper limit inequalities are generally achieved through density arguments.
Here, we avoid this by invoking an approximation result
of the total variation due to D\'avila \cite{dav:onanopen},
as it is also done by Maz\'on, Rossi and Toledo in \cite{mrt:nonlocalperimeter}.

Combining Theorems \ref{stm:cgl} and \ref{stm:Gammaconv}
we obtain a second $\Gamma$-convergence result:

	\begin{cor}
		Let us assume that \textbf{C1}, \textbf{C2'}, \textbf{C3} and \textbf{C4} hold.
		If for any measurable $u\colon \Rd \to [0,1]$
		we define the functionals
			\[\frac{1}{\eps}J_\epsilon(u,\Omega)\eqdef \frac{1}{\eps}J_{K_\eps}(u,\Omega)
			\quad\mbox{and}\quad J_0(u,\Omega)\eqdef c_K \abs{D u}(\Omega)\]
		then, as $\epsilon$ approaches $0$,
		the family $\set{\frac{1}{\eps}J_\epsilon(\,\cdot\,,\Omega)}$ $\Gamma$-converges
		to $J_0(\,\cdot\,,\Omega)$ w.r.t. the $L^1_\mathrm{loc}(\Rd)$ distance.
	\end{cor}
	\subsection{Rescaled nonlocal interactions and compactness} To deal with the proof of Theorem \ref{stm:Gammaconv},
we need some preliminary tools.
One of them is a compactness result
which appears rather natural in a $\Gamma$-convergence framework;
a second one is in fact more related to the peculiarities of our problem
and we discuss it in the lines that follow.
 
We point out that the functional $J^1_K(\,\cdot\,,\Omega)$ is not additive on disjoint subsets
w.r.t. its second argument
and this missing property accounts exactly for nonlocality.
Indeed, if $F$ is any measurable set
and we split the domain  $\Omega$ in the disjoint regions $\Omega\cap F$ and $\Omega\cap F^\co$,
for any  measurable $u\colon\Omega \to \R$, we get
	\begin{equation}\label{eq:nonlocJK1}
		\begin{split}
			J_K^1(u,\Omega) = & J_K^1(u,\Omega\cap F)+J_K^1(u,\Omega\cap F^\co) \\
							&+2\int_{\Omega\cap F}\int_{\Omega\cap F^\co} K(y-x)\abs{u(y)-u(x)}\de x\de y.
		\end{split}
	\end{equation}
The formula above shows that
the energy that is stored in two disjoint sets 
is smaller than the energy of their union
and that the difference is precisely given by the mutual interaction,
which, following the terminology suggested in \cite{ab:anonlocal}, 
we shall call \emph{locality defect}.

When one considers characteristic functions only, it can be easily seen that
the locality defect is the sum of certain nonlocal couplings;
hence, we are induced to analyse the limiting behaviour of nonlocal rescaled interactions.
Intuitively, since the kernel $K$ decays fast at infinity
and the operation of rescaling and letting the scaling parameter tend to $0$
amounts to ``concentrate'' the mass close to the origin, we expect some control
of the limit in terms of the portion of boundary shared by the two interacting sets.
The next statement puts this heuristic picture in precise terms:	

\begin{prop}[Asymptotic behaviour of nonlocal interactions]\label{stm:locdef}
	Let us consider $E,F\in\call{M}$.
	\begin{enumerate}
		\item If there exists a Caccioppoli set $E'$ in $\Rd$   
			such that $E\subset E'$ and $F\subset (E')^\co$, then
				\begin{equation*}
				\limsup_{\epsilon\to0}\frac{1}{\eps}L_\eps(E,F)\leq \frac{c_K'}{2}\Per(E').
				\end{equation*}
		\item If $\delta\eqdef\mathrm{dist}(E,F)>0$, then
				\begin{equation*}
				\lim_{\epsilon\to0}\frac{1}{\eps}L_\eps(E,F)=0.
				\end{equation*}
	\end{enumerate}
\end{prop}
\begin{proof}
	To prove the first estimate, we bound the interaction between $E$ and $F$
	by means of the interaction between $E'$ and its complement,
	that is, the nonlocal $K_\eps$-perimeter of $E'$:
		\[\begin{split}
			\frac{1}{\eps}L_\epsilon(E,F) &\leq\frac{1}{\eps}\int_{E'}\int_{(E')^\co}K_\epsilon(y-x)\de y \de x \\
					&=\frac{1}{2\eps}\int_{\Rd}\int_{\Rd}K(h)\abs{\chi_{E'}(x+\eps h)-\chi_{E'}(x)}\de h \de x \\
					&\leq\frac{1}{2}\int_{\Rd} K(h)\abs{h}\de h\Per(E'),
		\end{split}\]
	where the last inequality is obtained by Proposition \ref{stm:charBV}.
	
	Now, let us suppose that $\delta\eqdef\mathrm{dist}(E,F)>0$. Then
		\[\begin{split}
			\frac{1}{\eps}L_\epsilon(E,F) &\leq\frac{1}{\eps\delta}\int_E\int_F K_\epsilon(y-x)\abs{y-x}\de y \de x \\
					&=\frac{1}{\delta}\int_{E}\int_{\Rd}K(h)\abs{h}\chi_F(x+\eps h)\de h \de x 
		\end{split}\]
	and we draw the conclusion applying Lebesgue's dominated convergence Theorem.
\end{proof}

The other tool we mentioned is a compactness criterion.
Before stating it, we premise a Lemma, whose proof consists of direct computations:

\begin{lemma}\label{stm:G-ineq}
	Let $G \in L^1(\mathbb{R}^d)$ be a positive function.
	Then, for any $u\in L^\infty(\Rd)$ it holds
		\begin{equation*}
		\int_{\Rd\times\Rd} (G*G)(h) \abs{u(x+h)-u(x)}\de h\de x \leq
		2 \norm{G}_{L^1(\Rd)} J_G(u,\Rd).
		\end{equation*}
	In particular, when $u$ is the characteristic function of a measurable set $E$,
		\begin{equation} \label{eq:G-ineq}
		\int_{\Rd\times\Rd} (G*G)(h) \abs{\chi_E(x+h)-\chi_E(x)}\de h\de x \leq
		4 \norm{G}_{L^1(\Rd)}\Per_G(E).
		\end{equation}
\end{lemma}

\begin{thm}[Compactness criterion]\label{stm:compact}
	For any $n\in\N$, let us consider $\epsilon_n>0$ and a measurable $E_n\subset \Omega$.
	If $\epsilon_n \to 0$ and
	\begin{equation*}
	\frac{1}{\eps_n}J_{\epsilon_n}^1(E_n, \Omega) \quad\mbox{is uniformly bounded,}
	\end{equation*}
	there exist a subsequence $\set{E_{n_k}}$ and a set $E$ with finite perimeter in $\Omega$
	such that $\set{E_{n_k}}$ converges to $E$ in $L^1(\Omega)$. 
\end{thm}
\begin{proof}
	To avoid inconvenient notation, in what follows we omit the index $n$
	and we write $E_\epsilon$ in place of $E_n$.
	
	The idea is to build a second sequence $\set{v_\eps}$ 
	that is asymptotically equivalent to $\set{E_\epsilon}$ in $L^1(\Rd)$,
	i.e. $\norm{v_\eps -\chi_{E_\eps}}_{L^1(\Rd)}=O(\eps)$,
	but that in addition has better compactness properties.
	To this purpose, we consider a positive function $\phi\in C^\infty_c(\Rd)$
	and we set $v_\eps\eqdef \phi_\eps \ast \chi_{E_\eps}$,
	where
		\[\phi_\epsilon(x)\eqdef\frac{1}{\left(\int_{\Rd}\phi(h)\de h\right)\epsilon^d}
					\phi\left(\frac{x}{\eps}\right).\]
	Notice that any $v_\epsilon$ is supported in some ball $B$ containing $\Omega$.
	Easy computations show that
		\begin{equation}\label{eq:asympteq}
		\int_{\Rd} \abs{v_\eps(x)-\chi_{E_\eps}(x)}\de x
			\leq \int_{\Rd}\int_{\Rd} \abs{\phi_\eps(h)}\abs{\chi_{E_\eps}(x+h)-\chi_{E_\eps}(x)}\de h\de x
		\end{equation}
	and
		\begin{equation}\label{eq:BVbound}
		\begin{split}
		\int_B \abs{\nabla v_\eps(x)}\de x & =\int_{\Rd} \abs{\nabla v_\eps(x)}\de x \\
				& \leq \int_{\Rd}\int_{\Rd}\abs{\nabla\phi_\eps(h)}\abs{\chi_{E_\eps}(x+h)-\chi_{E_{\epsilon}}(x)}\de h\de x
		\end{split}
		\end{equation}
	(to get the last bound we took advantage of the equality $\int_{\Rd}\abs{\nabla\phi_\eps}=0$);
	we claim that it is possible to choose $\phi$ in such a way that
	\eqref{eq:asympteq} yields asymptotic equivalence of the two sequences and that 
	\eqref{eq:BVbound} provides a uniform bound on the $\mathrm{BV}$-norm of $\set{v_\eps}$.
	If our claim is true, on one hand, up to extraction of subsequences,
	$v_\epsilon$ converges to some $v\in\mathrm{BV}(B)$ in $L^1(B)$;
	on the other, this $v$ must be the characteristic function of some $E\subset\Omega$,
	because it is a $L^1(\Rd)$ cluster point of $\chi_{E_\eps}$.
	This concludes the proof.
	
	Now let us show that the claim holds.
	Define the truncation operator
		\[T_1(s)\eqdef\begin{cases}
				s & \mbox{if } \abs{s}\leq 1 \\
				1 & \mbox{otherwise}		
		\end{cases}\]
	and the truncated kernel $G\eqdef T_1 \circ K\in L^1(\Rd)\cap L^\infty(\Rd)$.
	We observe that the convolution $G*G$ is positive and continuous
	and, therefore, we can build a positive $\phi\in C_c^{\infty}(\Rd)\setminus\set{0}$ such that
		\begin{equation}
		\phi  \leq G*G \quad \text{and} \quad |\nabla \phi | \leq G*G. 
		\end{equation}
	Let us set
		\[G_\epsilon(h)\eqdef \frac{1}{\epsilon^d}G\left(\frac{h}{\eps}\right).\]
	With this choice, from \eqref{eq:asympteq} and \eqref{eq:BVbound} we obtain
		\begin{equation}\label{eq:asympteq2}
		\int_{\Rd} \abs{v_\eps(x)-\chi_{E_\eps}(x)}\de x
		\leq \int_{\Rd}\int_{\Rd} \abs{G_\eps\ast G_\eps(h)}\abs{\chi_{E_\eps}(x+h)-\chi_{E_\eps}(x)}\de h\de x
		\end{equation}
	and
		\begin{equation*}
		\int_{\Rd} \abs{\nabla v_\eps(x)}\de x
		\leq \frac{1}{\eps}\int_{\Rd}\int_{\Rd}\abs{G_\eps\ast G_\eps(h)}\abs{\chi_{E_\eps}(x+h)-\chi_{E_{\epsilon}}(x)}\de h\de x.
		\end{equation*}
	Both the right-hand sides of these inequalities can be bounded above by Lemma \ref{stm:G-ineq};
	we detail the estimates for \eqref{eq:asympteq2} only, the others being identical.
	Thanks to \eqref{eq:G-ineq}, we have 
		\begin{align*}
		\int_{\Rd}\int_{\Rd} \abs{G_\eps\ast G_\eps(h)}&\abs{\chi_{E_\eps}(x+h)-\chi_{E_\eps}(x)}\de h\de x \\
			& \leq 4 \norm{G}_{L^1(\Rd)}\Per_{G_\eps}(E_\eps) \\
			& \leq 4 \norm{G}_{L^1(\Rd)}\Per_{K_\eps}(E_\eps) \\
			& = 4 \norm{G}_{L^1(\Rd)} \left(\frac{1}{2}J_\eps^1(E_\eps,\Omega)+J_\epsilon^2(E_\eps,\Omega)\right) \\
			& = 4 \norm{G}_{L^1(\Rd)} \left(\frac{1}{2}J_\eps^1(E_\eps,\Omega)+L_\eps(E_\eps,E_\epsilon^\co\cap\Omega^\co)\right)
		\end{align*}
	and hence, in view of the current hypotheses and of Proposition \ref{stm:locdef}, we deduce
		\[\int_{\Rd}\int_{\Rd} \abs{G_\eps\ast G_\eps(h)}\abs{\chi_{E_\eps}(x+h)-\chi_{E_\eps}(x)}\de h\de x = O(\eps),\]
	as desired.		
\end{proof}
	\subsection{Asymptotic behaviour on finite perimeter sets: the upper limit inequality}\label{ssec:uli}It is possible to describe the asymptotic behaviour of the functional $J_\epsilon(\,\cdot\,,\Omega)$
when it is evaluated on finite perimeter sets.
This also provides an insight about the upper limit inequality
that is to be discussed later on in this subsection.

We extend a result
by Maz\'on, Rossi and Toledo contained in \cite{mrt:nonlocalperimeter}:
differently from that work, here we are able to cope also with unbounded domains.

\begin{prop}\label{stm:pointconv} Let $\tilde \Omega$ be an open subset of $\Rd$
with Lipschitz boundary, not necessarily bounded, or the whole space $\Rd$
and let conditions \textbf{C2'} and \textbf{C3} be satisfied.  
Then, if $E$ is a finite perimeter set in $\tilde \Omega$
such that $E\cap \tilde\Omega$ is bounded, it holds
	\begin{equation}\label{eq:limsup-mrt}
	\lim_{\epsilon\to0}\frac{1}{2\eps}J^1_\epsilon(E,\tilde\Omega) = c_K\Per(E,\tilde\Omega)
	\end{equation}
and if $E$ is also a finite perimeter set in $\Rd$ we have
	\begin{equation}\label{eq:limJ2}
	\lim_{\epsilon\to0}\frac{1}{\eps}J^2_\epsilon(E,\tilde\Omega) = c_K\call{H}^{d-1}(\partial^\ast E\cap\partial\tilde\Omega). 
	\end{equation}
\end{prop}

The proof of the analogous result proposed in \cite{mrt:nonlocalperimeter}
only relies on the approximation
of the total variation of the gradient of a function
by means of weighted integrals of the difference quotient.
Precisely, for $\epsilon>0$, consider a collection
of positive functions $\bar\rho_\eps\colon [0,+\infty)\to[0,+\infty)$ such that
\begin{equation*}
\lim_{\epsilon\to0}\int_{\delta}^{+\infty}\bar{\rho}_\epsilon(r)r^{d-1}\de r = 0 \quad\mbox{for all } \delta>0;
\end{equation*}
also, define
$\rho_\epsilon(h)\eqdef\bar\rho_\epsilon(r)$ whenever $h\in\Rd$ and $\abs{h}=r$
and assume that
\[\int_{\Rd}\rho_\epsilon(h)\de h = 1.\]
In \cite{dav:onanopen}, D\'avila proved the following:

\begin{thm}\label{stm:davila}
	Let $\Omega$ and $\set{\rho_\eps}_{\epsilon>0}$ be as above.
	Then, for any $u\in\mathrm{BV}(\Omega)$
	\begin{equation}
	\lim_{\epsilon\to0}\int_\Omega\int_\Omega \rho_\epsilon(y-x)\frac{\abs{u(y)-u(x)}}{\abs{y-x}}\de y\de x =  
	\alpha_{1,d} \vass{\D u}(\Omega).
	\end{equation}
	with $\alpha_{1,d}$ as in \eqref{eq:h=hd}.
\end{thm}

\begin{proof}[Proof of Proposition \ref{stm:pointconv}]
Since $E\cap\tilde\Omega$ is bounded, there exists an open ball $B$
such that the closure of $E\cap\tilde\Omega$ is contained in $B$
and in particular $\mathrm{dist}(E\cap\tilde\Omega, B^\co)>0$.
By \eqref{eq:nonlocJK1},
	\[\frac{1}{\eps}J_\eps^1(E,\tilde\Omega) = \frac{1}{\eps}J_\eps^1(E,\tilde\Omega\cap B)
							+ \frac{1}{\eps}L_\epsilon(E\cap\tilde\Omega,E^\co\cap\tilde\Omega\cap B^\co); \]
and the second summand in the right-hand side is negligible when $\epsilon\to0$
thanks to Proposition \ref{stm:locdef}.

The previous reasoning shows that, as far as $J^1_\epsilon$ is concerned,
we can always suppose that $\tilde\Omega$ is an open and bounded subset with Lipschitz boundary,
so that we may invoke Theorem \ref{stm:davila}: if we set
	\[\rho_\eps(h)=\frac{\alpha_{1,d}}{2c_K}K_\eps(h)\abs{\frac{h}{\eps}}\]
the conclusion of that result reads
	\[\lim_{\epsilon\to0}\frac{\alpha_{1,d}}{2c_K\eps}J^1_\epsilon(E,\tilde\Omega) = \alpha_{1,d} \Per(E,\tilde\Omega), \]
that is \eqref{eq:limsup-mrt}.

To show that \eqref{eq:limJ2} holds,
we consider the nonlocal interaction associated with $K_\epsilon$ between $E$ and $E^\co$
and we decompose it according to the partition $\{\tilde \Omega,\tilde \Omega^\co\}$:
	\[\int_E\int_{E^\co}K_\epsilon(y-x)\de x \de y = 
	J^1_\epsilon(E,\tilde\Omega)+J^1_\epsilon(E,\tilde\Omega^\co) + 2J^2_\epsilon(E,\tilde\Omega)\]
since the topological boundary of $\tilde\Omega$ is $\call{L}^d$-negligible,
we are allowed to apply \eqref{eq:limsup-mrt} to each summand:
we obtain
	\[\lim_{\epsilon\to0}\frac{1}{\eps}J^2_\epsilon(E,\tilde\Omega) =
		c_K\Per(E) - c_K\Per(E,\tilde\Omega) - c_K\Per(E,\tilde\Omega^\co)\]
and the conclusion follows.
\end{proof}

\begin{oss}
As an immediate corollary of the last result, we get the characterisation
	\begin{equation}\label{eq:car-cK}
		c_K = \lim_{\eps\to 0}\frac{1}{2\eps}J_\epsilon^1(H,U).
	\end{equation}
Actually, we shall need further equivalent descriptions of $c_K$,
see Lemma \ref{stm:lemma-bK2}.		
\end{oss}

Now, we turn to the upper limit inequality.
We need to show that whenever $E$ is a measurable set
there exists a recovery family $\set{E_\eps}$, i.e.
a family that converges to $E$ in $L^1_\mathrm{loc}(\Rd)$ such that
	\[\limsup_{\eps\to 0}\frac{1}{\eps}J_\epsilon(E_\eps,\Omega) \leq c_K \Per(E,\Omega).\]
First of all, we can assume that $E$ has finite perimeter in $\Omega$,
otherwise any family that converges to $E$ is a recovery one.
Secondly, let us assume that $E$ is a Caccioppoli set in the whole space $\Rd$;
if we retain a \emph{transversality condition} for $E$ and $\Omega$,
that is $\call{H}^{d-1}(\partial^\ast E \cap \partial\Omega)=0$,
then we can invoke Proposition \ref{stm:pointconv}
to deduce that the choice $E_\eps = E$ for all $\epsilon>0$ defines a recovery family.
Hence, the proof of statement \ref{stm:Gammasup} in Theorem \ref{stm:Gammaconv} is concluded
if we show that the class of finite perimeter sets in $\Rd$
that are transversal to $\Omega$ is dense in energy.
This is the content of the next Lemma:

	\begin{lemma}
	Let $E$ be a finite perimeter set in $\Omega$.
	Then, there exists a family $\set{E_\epsilon}_{\epsilon>0}$ of sets with smooth boundaries
	such that $\call{H}^{d-1}(\partial E_\epsilon\cap\partial\Omega)=0$ and
	that $E_\eps \to E$ in $L^1_{\mathrm{loc}}(\Rd)$ and $\Per(E_\epsilon,\Omega)\to\Per(E,\Omega)$.
	\end{lemma}
	\begin{proof}
		By standard approximation results for finite perimeter sets,
		there exists a family $\set{F_\eps}_{\eps>0}$ of open sets with smooth boundaries
		such that $F_\eps$ converges to $E$ in $L^1(\Rd)$ 
		and that $\Per(F_\eps,\bar \Omega)$ converges to $\Per(E,\Omega)$.
		Also, notice that in the family
		\[F_{\epsilon,t}\eqdef\set{x : \mathrm{dist}(x,F_\epsilon)-\mathrm{dist}(x,F_\epsilon^\co)\leq t}\]
		there must be some $E_\eps \eqdef F_{\epsilon,t^*}$
		which is smooth, transversal to $\Omega$
		and close to $F_\eps$ in $L^1$ and in perimeter.
	\end{proof}

	\subsection{Density estimates: the lower limit inequality}\label{ssec:lli} Let us focus on the proof of statement \ref{stm:Gammainf} in Theorem \ref{stm:Gammaconv}.
Given $E\in\call{M}$ and any family $\set{E_\epsilon}_{\epsilon>0}$
that converges to $E$ in $L^1_\mathrm{loc}(\Rd)$,
we want to show that
	\begin{equation}\label{eq:Gammainf}
	c_K\Per(E,\Omega)\leq \liminf_{\epsilon\to0} \frac{1}{2\eps}J_\epsilon^1(E_\eps,\Omega).
	\end{equation}
Observe that we can assume that the right-hand side is finite,
otherwise the inequality holds trivially, and
that the lower limit is a limit.
Therefore, the ratios $\frac{1}{\eps}J_\eps^1(E_\eps,\Omega)$ are bounded
and, in view of Theorem \ref{stm:compact},
$E$ is a Caccioppoli set in $\Omega$.

The first step of the approach \textit{\`a la} Fonseca-M\"uller
amounts to reducing the proof of \eqref{eq:Gammainf}
to the validity of a suitable density estimate.
To this aim, we introduce a family of positive measures
such that the total variation on $\Omega$ of each of them
is equal to $\frac{1}{2\eps}J_\epsilon^1(E_\eps,\Omega)$.
Namely, when $x\in\Omega$, let us set 
	\[ f_\eps(x) \eqdef \begin{cases}
				\displaystyle{\frac{1}{2\epsilon}\int_{E_\eps^\co\cap\Omega}K_\eps(y-x)\de y} & \text{if } x\in E_\eps \\
								\\
				\displaystyle{\frac{1}{2\epsilon}\int_{E_\eps\cap\Omega}K_\eps(y-x)\de y} & \text{if } x\in E_\eps^\co
		\end{cases}
	\]
and $\nu_\eps \eqdef f_\eps\,\call{L}^d \llcorner \Omega$.
In this way,
	\[\norm{\nu_\eps}\eqdef \abs{\nu_\eps}(\Omega)=\frac{1}{2\eps}J_\epsilon^1(E_\eps,\Omega);\]
since the right-hand side in the latter equality is uniformly bounded w.r.t. $\epsilon$,
we deduce that there exists a finite positive measure $\nu$ on $\Omega$
such that $	\nu_\eps\rightharpoonup^* \nu$ as $\epsilon \to 0$
and hence
	\[\liminf_{\epsilon\to0}\norm{\nu_\eps}\geq \norm{\nu}.\]
Because of the inequality above,
the conclusion \eqref{eq:Gammainf} follows if we prove that
	\begin{equation}\label{eq:tvineq}
		\norm{\nu}\geq c_K \Per(E,\Omega);
	\end{equation}
recalling \eqref{eq:per-Haus},
if we denote by $\mu$ the perimeter measure of $E$,
we see that \eqref{eq:tvineq} in turn is implied by
	\begin{equation}\label{eq:RNderineq}
	 \der{\nu}{\mu} (x) \geq c_K \quad \text{for } \mu\text{-a.e. } x \in \Omega,
	\end{equation}
where the left-hand side is the Radon-Nikodym derivative of $\nu$ w.r.t. $\mu$.
Summing up, the proof is concluded
if we show that \eqref{eq:RNderineq} holds.
This can be done by recovering at first a ``natural'' bound for the derivative (Lemma \ref{stm:lemma-bK1})
and then by proving that this bound is indeed the desired one (Lemmma \ref{stm:lemma-bK2}).

	\begin{lemma}\label{stm:lemma-bK1}
	Keeping the assumptions and the notation above,
	it holds
		\begin{equation*}
		\der{\nu}{\mu} (x) \geq b_K \quad\mbox{for every } x \in \partial^\ast E\cap\Omega.
		\end{equation*}
	where		
		\begin{equation}\label{eq:bK}
			b_K=\inf\set{\liminf_{\epsilon\to0} \frac{1}{2\eps} J^1_\eps\left(E_\eps,U\right) :
					E_\eps \to H \mbox{ in } L^1(U)}.
		\end{equation}
	\end{lemma}
	\begin{proof}
		Let us fix $x\in\partial^\ast E$.
		By \eqref{eq:per-Haus}, we have
			\[\der{\nu}{\mu}(x) =\lim_{r\to 0}\frac{\nu(Q(x,r))}{r^{d-1}},\]
		where $Q(x,r)\eqdef x+rR_x U$ and $R_{x}$ is chosen as to satisfy \eqref{eq:blowup}.
		Also, since the sequence $\nu_\epsilon$ weakly-$\ast$ converges to $\nu$,
		we have that $\nu(Q(x,r))=\lim_{\epsilon\to0}\nu_\eps(Q(x,r))$
		for all $r>0$ except at most a countable set $Z$ and hence
			\[\der{\nu}{\mu}(x) =\lim_{r\to 0, r\notin Z}\left[\lim_{\epsilon\to0}\frac{\nu_\eps(Q(x,r))}{r^{d-1}}\right].\]
		Via a diagonal process, it is possible to choose
		two sequences $\set{\eps_n}$ and $\set{r_n}$ such that
			\[\limsucc{n}r_n=\limsucc{n}\frac{\eps_n}{r_n}=0\]
		and that
			\[\der{\nu}{\mu}(x)=\limsucc{n}\frac{\nu_{\epsilon_n}(Q(x,r_n))}{r^{d-1}_n},\]
		or, explicitly,
			\[\begin{split}
				\der{\nu}{\mu}(x)=
				\limsucc{n}\frac{1}{2\eps_n r^{d-1}_n}
				&\left[\int_{E_{\epsilon_n} \cap Q(x,r_n)\cap \Omega}\int_{E_{\eps_n}^\co\cap\Omega}K_{\eps_n}(y-x)\de y\de x \right.\\ &\left.+\int_{E_{\epsilon_n}^\co \cap Q(x,r_n)\cap \Omega}\int_{E_{\eps_n}\cap\Omega}K_{\eps_n}(y-x)\de y. \de x\right]
			\end{split}\]
		From this equality we infer the lower bound
			\[\begin{split}
				\der{\nu}{\mu}(x)\geq
				\limsup_{\diverge{n}}\frac{1}{2\eps_n r^{d-1}_n}J^1_{\epsilon_n}(E_{\epsilon_n},Q(x,r_n)\cap \Omega) \\
				= \limsup_{\diverge{n}}\frac{1}{2\eps_n r^{d-1}_n}J^1_{\epsilon_n}(E_{\epsilon_n},Q(x,r_n))
			\end{split}\]
		(when $r_n$ is small enough, $Q(x,r_n)\subset \Omega$);
		moreover, by means of a change of variables and \eqref{eq:rotinv}, we find
			\[J^1_{\epsilon_n}(E_{\epsilon_n},Q(x,r_n))
				=r_n^d J^1_{\frac{\epsilon_n}{r_n}}\left(R_x^{-1}\left(\frac{E_{\epsilon_n}-x}{r_n}\right),U\right)
				\]
		and this, plugged in the last inequality, yields
			\[\der{\nu}{\mu}(x)\geq \limsup_{\diverge{n}}\frac{r_n}{2\eps_n}
					J^1_{\frac{\epsilon_n}{r_n}}\left(R_x^{-1}\left(\frac{E_{\epsilon_n}-x}{r_n}\right),U\right).\]
		Now, thanks to our choice of $R_x$ we have that
			\[R_x^{-1}\left(\frac{E_{\epsilon_n}-x}{r_n}\right)\to H \quad \mbox{in } L^1(U)\mbox{ as }\diverge{n},\]
		and by definition of $b_K$ we conclude.	
	\end{proof}
	
	To accomplish the proof of the inferior limit inequality in Theorem \ref{stm:Gammaconv}
	we have to show that $c_K = b_K$.
	Imitating the approach of \cite{adm:gammaconvergence},
	we do this by introducing a third constant $b'_K$:
	
	\begin{lemma}\label{stm:lemma-bK2}
		For any $\delta >0$, set $U^\delta \coloneqq \set{x\in U, \mathrm{dist}(x,U^\co)\leq\delta}$ and
			\[b'_K \eqdef \inf\set{\liminf_{\epsilon\to0} \frac{1}{2\eps} J^1_\eps\left(E_\eps,U\right) :
				E_\eps \to H \mbox{ in } L^1(U) \mbox{ and } E_\epsilon\cap U^\delta = H \cap U^\delta}.\]
		Under the previous assumptions, $c_K = b'_K = b_K$.			
	\end{lemma}
	
	Because of \eqref{eq:car-cK}, one clearly has $c_K \geq b'_K \geq b_K$.
	The conclusion of Lemma \ref{stm:lemma-bK2} follows
	if we prove that the reverse inequalities hold as well;
	this can be achieved invoking Proposition \ref{stm:flatness} and
	the next result, which extends a similar one
	proved in \cite{adm:gammaconvergence} for $s$-perimeters:
		
		\begin{prop}[Gluing]\label{stm:gluing}
			Let us consider $E_1,E_2\in\call{M}$
			and $\delta_1,\delta_2\in\R$ such that $\delta_1 > \delta_2>0$.
			For $\delta>0$, we set
				\[\Omega^{\delta}\eqdef\set{x\in\Omega : \mathrm{dist}(x,\Omega^\co)\leq \delta}.\]
			If $J^1_K(E_i,\Omega)$ is finite for both $i=1,2$, then
			there exists $F\in\call{M}$ such that 
			\begin{enumerate}
				\item\label{stm:glue1} $F\cap(\Omega\setminus \Omega^{\delta_1}) = E_1 \cap (\Omega\setminus \Omega^{\delta_1})$
					and $F\cap \Omega^{\delta_2} = E_2 \cap \Omega^{\delta_2}$;
				\item\label{stm:glue2} $\abs{(E_1\symdif F)\cap\Omega}\leq \abs{(E_1\symdif E_2)\cap\Omega}$;
				\item for all $\eta >0$:
					\begin{equation}\label{eq:gluing}
						\begin{split}
							J_\eps^1(F,\Omega) \leq & J_\eps^1(E_1,\Omega) + J_\eps^1(E_2,\Omega^{\delta_1+\eta}) 
								+ \frac{2 \eps c'_K}{\delta_1-\delta_2}\vass{(E_1\symdif E_2)\cap\Omega} \\
								& + \frac{2 \eps}{\eta}\int_{\Omega\setminus\Omega^{\delta_1+\eta}}\int_{\Rd}
								K(h)\vass{h}\chi_{\Omega^{\delta_1}}(x+\eps h)\de h\de x.
						\end{split}
					\end{equation}
			\end{enumerate}
		\end{prop}
		\begin{proof}
			Suppose that for some function $w\colon \Omega \to [0,1]$ it holds
				\begin{equation}\label{eq:gluing2}
				\begin{split}
				J_\eps^1(w,\Omega) \leq & J_\eps^1(E_1,\Omega) + J_\eps^1(E_2,\Omega^{\delta_1+\eta}) 
								+ \frac{2 \eps c'_K}{\delta_1-\delta_2}\vass{(E_1\symdif E_2)\cap\Omega} \\
								& + \frac{2 \eps}{\eta}\int_{\Omega\setminus\Omega^{\delta_1+\eta}}\int_{\Rd}
								K(h)\vass{h}\chi_{\Omega^{\delta_1}}(x+\eps h)\de h\de x;
				\end{split}
				\end{equation}
			then, thanks to Coarea formula there exists $t^\ast\in(0,1)$
			such that \eqref{eq:gluing} holds for the superlevel $F\eqdef\set{w>t^\ast}$.
			Let us exhibit a function $w$ that fulfils \eqref{eq:gluing2}.
			
			Loosely speaking, we choose $w$ to be a convex combination
			of the data $\chi_{E_1}$ and $\chi_{E_2}$.
			Precisely, for any $u,v\colon \Omega\to[0,1]$ such that
			$J^1_K(u,\Omega)$ and $J^1_K(v,\Omega)$ are finite, let us set
			$w\eqdef\phi u + (1-\phi)v$, where $\phi\in C^{\infty}_c(\Rd)$ satisfies 
				\[0\leq \phi \leq 1 \text{ in } \Omega,\quad
				\phi=0 \text{ in } \Omega^{\delta_2}, \quad
				\phi =1 \text{ in } \Omega\setminus \Omega^{\delta_1} \quad\text{and}\quad
				\abs{\nabla \phi} \leq \frac{2}{\delta_1-\delta_2}.\]	
			We explicit the integrand appearing in $J^1_K(w,\Omega)$ and
			for $x,y\in \Omega$ we get the bounds
				\[\begin{split}
					\vass{w(y) - w(x)} \leq & \phi(y)\vass{u(y)-u(x)} + (1-\phi(y))\vass{v(y)-v(x)} \\
											& + \vass{\phi(y)-\phi(x)}\vass{v(x)-u(x)} \\
								\leq & \chi_{\set{\phi \neq 0}}(y)\vass{u(y)-u(x)} + \chi_{\set{\phi \neq 1}}(y)\vass{v(y)-v(x)} \\
											& + \vass{\phi(y)-\phi(x)}\vass{v(x)-u(x)}.
				\end{split}\]
			By our choice of $\phi$,
			$\set{\phi \neq 0} \subset \Omega\setminus \Omega^{\delta_2}$ and
			$\set{\phi \neq 1} \subset \Omega^{\delta_1}$, therefore
				\[\begin{split}
				J_\eps^1(w, \Omega)\leq & \int_\Omega\int_{\Omega\setminus \Omega^{\delta_2}} K_\eps(y-x)\vass{u(y)-u(x)}\de y\de x \\
									& + \int_\Omega\int_{\Omega^{\delta_1}} K_\eps(y-x)\vass{v(y)-v(x)}\de y\de x \\
									& + \int_\Omega\int_\Omega K_\eps(y-x)\vass{\phi(y)-\phi(x)}\vass{v(x)-u(x)}\de y\de x
				\end{split}\]
			We treat each of the integrals appearing in the right-hand side separately.
			Of course one has
				\begin{equation}\label{eq:I1}
				\int_\Omega\int_{\Omega\setminus \Omega^{\delta_2}} K_\eps(y-x)\vass{u(y)-u(x)}\de y\de x \leq J_\eps^1(u,\Omega).
				\end{equation}
			To estimate the second integral,
			we split the reference set $\Omega$ in the regions
			$\Omega^{\delta_1+\eta}$ and $\Omega\setminus\Omega^{\delta_1+\eta}$
			and this yields
				\[\begin{split}
					\int_\Omega\int_{\Omega^{\delta_1}}& K_\eps(y-x)\vass{v(y)-v(x)}\de y\de x = \\
						& \int_{\Omega^{\delta_1+\eta}}\int_{\Omega^{\delta_1}} K_\eps(y-x)\vass{v(y)-v(x)}\de y\de x \\
						& +\int_{\Omega\setminus\Omega^{\delta_1+\eta}}\int_{\Omega^{\delta_1}} K_\eps(y-x)\vass{v(y)-v(x)}\de y\de x;
				\end{split}\]
			evidently
				\[\int_{\Omega^{\delta_1+\eta}}\int_{\Omega^{\delta_1}} K_\eps(y-x)\vass{v(y)-v(x)}\de y\de x
					\leq J_\eps^1(v,\Omega^{\delta_1+\eta})\]
			and, further,
				\[\begin{split}
					\int_{\Omega\setminus\Omega^{\delta_1+\eta}}\int_{\Omega^{\delta_1}}& K_\eps(y-x)\vass{v(y)-v(x)}\de y\de x \\
					\leq  &	\frac{2\eps}{\eta}\int_{\Omega\setminus\Omega^{\delta_1+\eta}}\int_{\Omega^{\delta_1}}
							K_\eps(y-x)\frac{\vass{y-x}}{\eps}\de y\de x \\
					\leq  & \frac{2\eps}{\eta}\int_{\Omega\setminus\Omega^{\delta_1+\eta}}\int_{\Rd}
							K(h)\vass{h}\chi_{\Omega^{\delta_1}}(x+\eps h)\de h\de x,
				\end{split}\]
			so that, all in all,
				\begin{equation}\label{eq:I2}
				\begin{split}
				\int_\Omega\int_{\Omega^{\delta_1}} & K_\eps(y-x)\vass{v(y)-v(x)}\de y\de x \\
						\leq &  J_\eps^1(v,\Omega^{\delta_1+\eta}) 
							+\frac{2\eps}{\eta}\int_{\Omega\setminus\Omega^{\delta_1+\eta}}\int_{\Rd}
							K(h)\vass{h}\chi_{\Omega^{\delta_1}}(x+\eps h)\de h\de x.
				\end{split}
				\end{equation}
			Lastly,	we observe that
				\[\vass{\phi(y)-\phi(x)}\leq \frac{2}{\delta_1-\delta_2}\vass{y-x}\]
			and hence
				\begin{equation}\label{eq:I3}
				\int_\Omega\int_\Omega K_\eps(y-x)\vass{\phi(y)-\phi(x)}\vass{v(x)-u(x)}\de y\de x \leq
					\frac{2 \eps c'_K}{\delta_1-\delta_2}\int_\Omega \vass{v(x)-u(x)}\de x
				\end{equation}
			Combining \eqref{eq:I1}, \eqref{eq:I2} and \eqref{eq:I3}
			we obtain
				\[\begin{split}
					J_\eps^1(w,\Omega)\leq & J_\eps^1(u,\Omega)+J_\eps^1(v,\Omega^{\delta_1+\eta}) 
							+ \frac{2 \eps c'_K}{\delta_1-\delta_2}\norm{v-u}_{L^1(\Omega)} \\
							& + \frac{2 \eps}{\eta}\int_{\Omega\setminus\Omega^{\delta_1+\eta}}\int_{\Rd}
								K(h)\vass{h}\chi_{\Omega^{\delta_1}}(x+\eps h)\de h\de x ;
				\end{split}\]
			thus, if we pick $u= \chi_{E_1}$ and $v= \chi_{E_2}$
			we have \eqref{eq:gluing2}.
				
			It remains to check that conditions \ref{stm:glue1} and \ref{stm:glue2} hold true
			for the set $F$ defined above.
			As for the former, it suffices to recall that
			$\phi$ is supported in $\Omega\setminus\Omega^{\delta_2}$
			and that it is constantly $1$ in $\Omega\setminus\Omega^{\delta_1}$.
			On the other hand, to prove the second condition we remark that
			$x\in E_1\cap F^\co$ and $x\in E_1^\co\cap F$ imply respectively the equalities
			$w(x)=\phi(x)+(1-\phi(x))\chi_{E_2}(x)\leq t < 1$ and $w(x)=(1-\phi)\chi_{E_2}(x)>t>0$, 
			which in turn entail $x\in E_1\cap E^\co$ and $x\in E_1^\co \cap E_2$.
		\end{proof}
	
	\begin{proof}[Proof of Lemma \ref{stm:lemma-bK2}]
	We firstly show that $c_K \leq b'_K$.
	Choose arbitrarily $\delta >0$ and
	consider a family $\set{E_\epsilon}$ such that 
	$E_\eps \to H$ in $L^1(U)$ and $E_\epsilon\cap U^\delta = H \cap U^\delta$.
	We may extend each $E_\epsilon$ outside the unit cube
	putting $E_\eps \cap U^\co = H \cap U^\co$.
	Then, we invoke the third statement of Corollary \ref{stm:flatness} to infer
	$J_\eps(H,U)\leq J_\eps(E_\epsilon,U)$, which gets
		\[\frac{1}{2\eps}J^1_\eps(H,U) \leq \frac{1}{2\eps}J^1_\eps(E_\epsilon,U)
								+ \frac{1}{\eps}\left(J^2_\eps(E_\epsilon,U) - J^2_\eps(H,U)\right).\]
	The desired inequality follows if we show that
	the second summand in the right-hand side vanishes as $\epsilon$ tends to $0$.
	To this aim, we exploit our information about $\set{E_\epsilon}$,
	which provides the equalities
		\[\begin{split}
		J^2_\eps(E_\epsilon,U) & - J^2_\eps(H,U) \\
			= & L_\epsilon(E_\epsilon\cap U,H^\co\cap U^\co) + L_\epsilon(H\cap U^\co,E_\epsilon^\co\cap U) \\
				& - L_\epsilon(H\cap U,H^\co\cap U^\co) - L_\epsilon(H\cap U^\co,H^\co\cap U) \\
			= & L_\epsilon(E_\epsilon\cap (U\setminus U^\delta),H^\co\cap U^\co)
				+ L_\epsilon(H\cap U^\co,E_\epsilon^\co\cap (U\setminus U^\delta)) \\
				& - L_\epsilon(H\cap (U\setminus U^\delta),H^\co\cap U^\co)
				- L_\epsilon(H\cap U^\co,H^\co\cap (U\setminus U^\delta));
		\end{split}\]
	and from this, in view of Proposition \ref{stm:locdef}, the claim is proved.
	
	Now, we show that $b'_K \leq b_K$.
	We let $\set{E_\epsilon}$ be such that 
	$E_\eps \to H$ in $L^1(U)$ as $\epsilon$ approaches $0$ and,
	without loss of generality, that $J_{\epsilon}^1(E_{\epsilon},U)$ is finite.   
	For any $\epsilon$, we apply Proposition~\ref{stm:gluing} to $E_\epsilon$ and $H$
	and this yields a family $\set{F_\epsilon}$ with the properties that
	it $L^1$-converges to $H$ in $U$, that $F_\epsilon\cap U^\delta = H\cap U^\delta$ and that for any $\eta>0$
		\[\begin{split}
			\frac{1}{2\eps}J_\eps^1(F_\eps,U) \leq & \frac{1}{2\eps}J_\eps^1(E_\eps,U)
							+ \frac{1}{2\eps}J_\eps^1(H,U^{\delta+\eta}) 
					+ \frac{2c'_K}{\eta}\vass{(E_\eps \symdif H)\cap U} \\
					& + \frac{1}{\eta}\int_{U\setminus U^{\delta_1+\eta}}\int_{\Rd}
						K(h)\vass{h}\chi_{U^{\delta_1}}(x+\eps h)\de h\de x.
		\end{split}\]
	We notice that in view of \eqref{eq:limsup-mrt} it holds
		\[\lim_{\epsilon\to0}\frac{1}{2\eps}J_\eps^1(H,U^{\delta+\eta})=c_K \Per(H,U^{\delta+\eta})\]
	and consequently, taking the limit as $\epsilon\to0$, we get
		\[\begin{split}
			b'_K \leq & \liminf_{\epsilon\to0}\frac{1}{2\eps}J_\eps^1(F_\eps,U) \\
				 \leq & \liminf_{\epsilon\to0}\frac{1}{2\eps}J_\eps^1(E_\eps,U)
						+ c_K \Per(H,U^{\delta+\eta});
			\end{split}\]
	next, we let $\eta$ and $\delta$ vanish and,
	eventually, we take advantage of the arbitrariness
	of the family $\set{E_\eps}$ to conclude. 
	\end{proof}
	
\begin{oss}[Outlook]
The $\Gamma$-convergence result of Theorem \ref{stm:Gammaconv} might be extended
to different classes of kernels.
The first step would be dropping the assumption of strict monotonicity;
this entails the use of a different strategy
for the proof of the $\Gamma$-inferior limit inequality,
but the same conclusions would hold.
Similarly, we expect an analogous result without the radial symmetry hypothesis;
nevertheless, if $K$ is anisotropic, De Giorgi's perimeter has to be replaced with the functional
	\[\int_{\partial^\ast E} \sigma_K(\hat n(z))\de\call{H}^{d-1},\]
where the function $\sigma_K\colon \partial B(0,1) \to \R$ is a weight that depends on $K$ and
that corresponds to the constant $c_K$ appearing in this paper,
see the analysis carried out in \cite{ab:anonlocal} for ``localised'' functionals.

The cases when $K$ changes sign or when it is substituted by a Radon measure $\mu$
are also possible subjects of further study,
but the conclusions of Theorem \ref{stm:Gammaconv} might be affected.. 

Finally, it would be interesting to obtain a $\Gamma$-convergence result for multi-phase systems,
i.e. for functionals of the form
	\[J_\eps(E_1,\dots,E_N)\eqdef \sum_{0\leq i<j\leq N} a_{i,j}L_\eps(E_i,E_j),\]
where $E_1,\dots,E_N$ are measurable sets such that $\vass{E_i}>0$ and $\vass{E_i\cap E_i}=0$ for any $i,j=1,\dots,N$,
$E_0 \eqdef (\bigcup_{i=1}^N E_i)^\co$ and
the coefficients $a_{i,j}$ are positive and they satisfy $a_{i,j}\leq a_{i,k}+a_{k,j}$
for every $i,j,k=0,\dots,N$. 
As a particular instance,
if for any $i=0,\dots,N$ there exists $a_i\geq0$ such that
$a_i=a_{i,j}$ for all $j\neq i$, we get
	\[J_\eps(E_1,\dots,E_N)= \frac{1}{2}\sum_{i=0}^N a_{i}J_\epsilon(E_i,\Rd);\]
in this case, it easy to recover a $\Gamma$-convergence result from the theory we developed in this paper.
\end{oss}

\addcontentsline{toc}{section}{\refname}
\printbibliography
\end{document}